\documentclass[a4paper,12pt,reqno]{amsart}
\usepackage[colorlinks=true]{hyperref}
\usepackage{amssymb}
\usepackage{pstricks}
\usepackage{graphicx}

\allowdisplaybreaks[1]

\newtheorem{theorem}{Theorem}
\newtheorem{lemma}[theorem]{Lemma}
\newtheorem{corollary}[theorem]{Corollary}
\newtheorem{proposition}[theorem]{Proposition}
\newcounter{listnumber}
\newcommand{\F}{\mathcal{F}}
\newcommand{\G}{\mathcal{G}}
\renewcommand{\H}{\mathcal{H}}

\newcommand{\N}{\mathcal{N}}
\renewcommand{\P}{\mathcal{P}}
\renewcommand{\S}{\mathcal{S}}
\newcommand{\R}{\mathbb{R}}
\renewcommand{\vert}{\mathrm{vert}}
\newcommand{\V}{\mathcal{V}}
\newcommand{\cupdot}{\Cup}
\DeclareMathOperator{\nullity}{nullity}
\DeclareMathOperator{\rank}{rank}
\DeclareMathOperator{\supp}{supp}
\DeclareMathOperator{\facets}{facets}
\DeclareMathOperator{\conv}{conv}
\DeclareMathOperator{\gr}{gr}

\author{Roger E.~Behrend}
\address{Roger E.~Behrend, School of Mathematics, Cardiff University, Cardiff, CF24 4AG, UK}
\email{behrendr@cardiff.ac.uk}
\title[Fractional Perfect $b$-Matching Polytopes]{Fractional Perfect $b$-Matching Polytopes\\ I: General Theory}
\begin{document}
\begin{abstract}
The fractional perfect $b$-matching polytope of an undirected graph~$G$ is
the polytope of all assignments of nonnegative real numbers to the edges of~$G$
such that the sum of the numbers over all edges incident to any
vertex~$v$ is a prescribed nonnegative number~$b_v$.
General theorems which provide conditions for nonemptiness, give a formula for the dimension,
and characterize the vertices, edges and face lattices of such polytopes are obtained.
Many of these results are expressed in terms of certain spanning subgraphs of $G$ which are associated with subsets or elements
of the polytope.  For example, it is shown that an element~$u$ of the fractional perfect $b$-matching polytope of $G$ is a vertex of the
polytope if and only if each component of the graph of~$u$
either is acyclic or else contains exactly one cycle with that cycle having odd length,
where the graph of $u$ is defined to be the spanning
subgraph of~$G$ whose edges are those at which~$u$ is positive.
\end{abstract}
\keywords{graphs, polytope, perfect matching}
\subjclass[2010]{52B05; 05C50, 05C70, 52B11, 90C27, 90C35}
\maketitle
{\small\tableofcontents}
\newpage
\section{Introduction}
The focus of this paper, and its expected sequels~\cite{Beh13b,Beh13c}, is the fractional perfect $b$-matching polytope of a graph.
For any finite, undirected graph~$G$, which may contain loops and multiple edges, and
any assignment~$b$ of nonnegative real numbers to the vertices of~$G$, this polytope, denoted~$\P(G,b)$,
is defined to be the set of all assignments of nonnegative real numbers to the
edges of~$G$ such that the sum of the numbers over all edges incident to any
vertex is the prescribed value of $b$ at that vertex.

Certain fractional perfect $b$-matching polytopes,
or polytopes which are affinely isomorphic to these,
have been studied and used in the contexts of
combinatorial matrix classes (see, for example, the book by Brualdi~\cite{Bru06}), and
combinatorial optimization (see, for example, the books by Korte and Vygen~\cite{KorVyg12}, or Schrijver~\cite{Sch03}).
The terminology `fractional perfect $b$-matching polytope' is derived mainly from the latter context, and
will be discussed further in Section~\ref{relsect}.

\subsection{Results and structure of paper}\label{summ}
The primary aim of this paper is to present theorems concerning the nonemptiness, dimensions, vertices, edges
and faces of arbitrary such polytopes, together with uniform, and in most cases self-contained, proofs.

The main theorems apply to an arbitrary finite graph~$G$, which may be nonbipartite.
However (as is often the case with results related to matchings of graphs),
many of these theorems take simpler forms if~$G$ is bipartite, so these forms will also be given.

A list of the main results of this paper, including those which are restricted to the case of bipartite $G$, is as follows.
\begin{list}{$\bullet$}{\setlength{\labelwidth}{2mm}\setlength{\leftmargin}{4mm}\setlength{\itemsep}{1.4mm}}
\item\emph{Condition for $\P(G,b)$ to be nonempty:} Theorem~\ref{nonemptth}.\\
Bipartite case: Theorem~\ref{bipnonemptth}.
\item\emph{Results for the dimension of $\P(G,b)$:} Corollaries~\ref{dim} and~\ref{dimcor1}.
\item\emph{Results for the vertices of $\P(G,b)$:} Corollary~\ref{vertsizecor}, and
Theorems~\ref{vertth}, \ref{explvert} and~\ref{vertth2}.\\
Bipartite case: Corollary~\ref{bipvertcor}.
\item\emph{Results for the edges of $\P(G,b)$:} Theorem~\ref{edgeth} and Corollary~\ref{edgeth2}.\\
Bipartite case: Corollary~\ref{bipedgeth}.
\item\emph{Results for the faces of $\P(G,b)$}: Theorems~\ref{facedim}, \ref{faceth1}, \ref{faceth}, \ref{facelattth},
\ref{suppcond} and~\ref{bipsuppcond}.\\
Bipartite case: Corollary~\ref{bipsuppcond1}.
\end{list}

Many of these results are expressed in terms of certain spanning subgraphs of $G$
(i.e., subgraphs of~$G$ with the same vertex set as that of~$G$) which are associated with subsets or elements
of~$\P(G,b)$.  These spanning subgraphs, which will be referred to as the graphs of~$\P(G,b)$,
will be defined in~\eqref{grX} and~\eqref{grx},
and it will be found in Theorem~\ref{facelattth} that the set of all such graphs
forms a lattice which is isomorphic to the face lattice of~$\P(G,b)$.

The structure of the paper is as follows.  In Section~\ref{mat}, certain alternative forms,
involving the incidence matrix or a generalized adjacency matrix of $G$,
will be given for $\P(G,b)$.  In Section~\ref{nonempt}, results which give
conditions for $\P(G,b)$ to be nonempty, or
to contain an element whose value at each edge of $G$ is positive, will be derived.  In Section~\ref{inc}, some relevant general results for
graphs, and their incidence matrices, will be obtained.  In Sections~\ref{poly} and~\ref{facelatt}, some relevant general results
for polytopes will be obtained. In Section~\ref{facesPGb}, results for the faces, dimension, vertices and edges of $\P(G,b)$ will
be derived.  The proof of each of the main results of Section~\ref{facesPGb} will involve a relatively simple combination
of a general result for graphs from Section~\ref{inc} with a general result for polytopes from Section~\ref{poly}. In Section~\ref{facelattPGb},
further results for the vertices, edges, faces and graphs of $\P(G,b)$ will be derived.  The proofs of each of the main results of Section~\ref{facelattPGb}
will involve a relatively simple application of a general result for polytopes from Section~\ref{facelatt} to the context of $\P(G,b)$.  Finally,
in Section~\ref{furthcond}, some additional results for the graphs of $\P(G,b)$ will be obtained, using certain results from Section~\ref{nonempt}.

For the sake of completeness, this paper includes some previously-known results.
However, many of these results have
appeared previously in the literature only under slightly less general conditions (for example, graphs without
loops or multiple edges), in terms of slightly different objects (for example, matrices
rather than graphs), or with somewhat different proofs (for example, those in which the aspects which depend on graph theory
and the aspects which depend on polytope theory are interspersed throughout the proof, rather than being considered separately until the final step).

\subsection{Conventions and notation}\label{convnot}
The main conventions and notation (and some standard facts), which will be used in this paper are as follows.

Throughout the paper, $G$ is a finite, undirected graph which, unless stated otherwise,
may be nonbipartite, and may contain loops and multiple edges.  Furthermore,~$V$ and~$E$ are the vertex
and edge sets of~$G$, and~$b$ is a function from~$V$ to the nonnegative real numbers.

It will always be assumed that $V$ is nonempty, but, unless stated otherwise, that $E$ may be empty.

The set of all edges incident with vertex~$v$ of~$G$ will be denoted as~$\delta_G(v)$,
with the implication that a loop attached to $v$ appears once rather than twice in~$\delta_G(v)$.

A cycle in~$G$ of length~1 is taken to be a vertex together with a loop attached to that vertex,
and a cycle in~$G$ of length~2 is taken to be two distinct vertices together with two distinct edges, each of which connects those
vertices.
With these conventions, the following standard facts,
which are often applied only to simple graphs (i.e., those without loops or multiple edges), are valid.
The graph~$G$ is bipartite if and only if~$G$ does not contain any odd-length cycles.
For a connected graph~$G$, $|E|+1=|V|$ if and only if $G$ is acyclic (i.e., a tree),
and $|E|=|V|$ if and only if $G$ contains exactly one cycle.

For subsets $U$ and $W$ of $V$, the set of all edges of~$G$ which connect a vertex
of~$U$ and a vertex of~$W$ will be denoted as $G[U,W]$, i.e.,
\begin{equation}G[U,W]:=\{e\in E\mid\text{the endpoints of }e\text{ are  }u\text{ and }w,\text{ for some }u\in U\text{ and }w\in W\}.\end{equation}
Some simple properties of such sets are that,
for any $U_1,U_2,U_3\subset V$, $G[U_1,U_2]=G[U_2,U_1]$ and $G[U_1,U_2\cup U_3]=G[U_1,U_2]\cup G[U_1,U_3]$.
Also, a vertex cover of $G$ (i.e., a subset of $V$ which contains an endpoint of every edge of $G$)
is any $C\subset V$ for which $G[V\setminus C,V\setminus C]=\emptyset$.

The incidence matrix of $G$ will be denoted as $I_G$.  See also the beginning of Section~\ref{mat}.

The partitioning of a set into a union of finitely-many pairwise disjoint subsets will be expressed using
the notation $\cupdot$.  More specifically, for sets $U,W_1,W_2,\ldots,W_n$, the statement
$U=W_1\cupdot W_2\cupdot\ldots\cupdot W_n$ will mean that $U=W_1\cup W_2\cup\ldots\cup W_n$
and $W_i\cap W_j=\emptyset$ for each $i\ne j$.

The notation $\subset'$ will be defined near the beginning of Section~\ref{facelatt}.

Matrices and vectors whose rows and columns are indexed by finite sets will often be used.  If such a matrix or vector is
written out as an explicit array of entries, then an ordering of the elements of each associated index set needs to be chosen.
However, all of the results of this paper are independent of these choices, and hence no such choices will be made.

The rank and nullity of a real matrix $A$, with respect to the field $\R$, will be denoted as $\rank(A)$ and $\nullity(A)$, respectively.

For a finite set $N$, the vector space of all functions from~$N$ to~$\R$ will be denoted as~$\R^N$, with~$\R^\emptyset$ taken to be $\{0\}$.
The value of $x\in\R^N$ at~$i\in N$ will be denoted as $x_i$,
so that~$x$ is regarded as a vector whose entries are
indexed by $N$.  A vector $x\in\R^N$ will be called strictly positive if $x_i>0$ for each $i\in N$.
For $X\subset\R^N$ and $x\in\R^N$, the supports of~$X$ and~$x$ will be denoted as
$\supp(X)$ and $\supp(x)$. See~\eqref{supp}--\eqref{suppx} for definitions. The convex hull
of $X$ will be denoted as $\conv(X)$.

The set of vertices, set of facets and
face lattice of a polytope~$P$ will be denoted as $\vert(P)$, $\facets(P)$ and~$\F(P)$, respectively.

The fractional perfect $b$-matching polytope of $G$ can now be written, using some of the notation above, as
\begin{equation}\label{PGb}\P(G,b):=\biggl\{x\in\R^E\biggm|x_e\ge0\mbox{ for
each }e\in E,\ \ \sum_{e\in\delta_G(v)}\!\!x_e=b_v\mbox{ for each }v\in V\biggr\}.\end{equation}

This set is a polytope in $\R^E$ since
it is a polyhedron in $\R^E$ (being the intersection of the closed halfspaces $\{x\in\R^E\mid x_e\ge0\}$ for each $e\in E$, and
the hyperplanes $\{x\in\R^E\mid\sum_{e\in\delta_G(v)}x_e=b_v\}$ for each $v\in V$),
and it is bounded (since any $x\in\P(G,b)$ satisfies $0\le x_e\le b_{v_e}$ for each $e\in E$,
where $v_e$ is an endpoint of $e$).

It will be assumed, for some of the results of this paper, that~$b$ is nonzero.  It can be seen that this is equivalent to the assumption
that $\P(G,b)\ne\{0\}$.

For the case $E=\emptyset$, $\P(G,b)$ is taken to be $\R^\emptyset=\{0\}$ if $b=0$, or to be $\emptyset$ if $b\ne 0$.

For $X\subset\P(G,b)$ and $x\in\P(G,b)$, the graphs of~$X$ and~$x$ (as already introduced briefly in
Section~\ref{summ}) will be denoted as
$\gr(X)$ and $\gr(x)$. See~\eqref{grX}--\eqref{grx} for definitions.
The set of graphs of $\P(G,b)$ will be denoted as $\G(G,b)$. See~\eqref{GGbdef} for a definition,
and~\eqref{GGb} and~\eqref{GGb2} and for several further characterizations.

\subsection{Related matching and polytope types}\label{relsect}
In order to place fractional perfect $b$-matching polytopes into a wider context, some related matchings and polytopes
will now be defined.  However, this information is not needed in the remainder of this paper.

For a graph $G$, and a vector $b\in\R^V$ with all entries nonnegative,
define eight types of vector $x\in\R^E$ with all entries nonnegative,
which satisfy
\begin{equation}\label{match}\sum_{e\in\delta_G(v)}\!\!x_e\le b_v\text{ \ for each }v\in V,\end{equation}
where the types are obtained by prefixing any of the terms `fractional', `perfect' or `$b$-' to the term `matching',
and these terms have the following meanings.  If `fractional' is omitted, then all entries of $b$ and $x$ are integers.
If `perfect' is included, then each case of~\eqref{match} holds as an equality.
If `$b$-' is omitted, then each entry of $b$ is 1.

For each of these types of matching, define an associated polytope as the set of all such matchings if `fractional'
is included, or as the convex hull of all such matchings if `fractional' is omitted.
Thus, for example, this definition of the fractional perfect $b$-matching polytope coincides with the previous definition~\eqref{PGb}.
For more information on the other cases of such matchings and
polytopes (and further related special cases, such as those which are `capacitated' or `simple'), see, for example, Korte and Vygen~\cite{KorVyg12},
Lov{\'a}sz and Plummer~\cite{LovPlu09}, or Schrijver~\cite{Sch03}).  Some of these cases
will also be considered in~\cite{Beh13b}.

Note that, using the previous definitions, a matching or perfect matching $x$ is an assignment $x_e$ of~$0$ or~$1$
to each edge $e$ of $G$ such that the sum of the numbers over all edges incident to any
vertex is at most~1, or exactly~1 in the perfect matching case.
However, using standard graph theory terminology, such an $x$ is actually the incidence vector of a matching
(i.e., a subset $M$ of $E$ such that each vertex is incident to at most one edge in~$M$),
or perfect matching (i.e., a subset $M$ of $E$ such that each vertex is
incident to exactly one edge in~$M$).
More specifically, $x_e$ is $1$ or $0$ according to whether or not the edge $e$ is in the matching.

\subsection{Further papers}
This is the first paper in a projected series of three papers on fractional perfect $b$-matching polytopes.

In the second paper~\cite{Beh13b},
various polytopes which are special cases of fractional perfect $b$-matching polytopes,
or which are affinely isomorphic to such special cases,
will be considered, and results (including certain standard theorems)
for these cases will be obtained by applying the general theorems
of this paper.  The cases which will be considered in~\cite{Beh13b} will include the following.
\begin{list}{$\bullet$}{\setlength{\labelwidth}{2mm}\setlength{\leftmargin}{5mm}}
\item Polytopes $\P(G,b)$ in which each entry of $b$ is an integer.
\item Polytopes defined by modifying~\eqref{PGb} so that $\sum_{e\in\delta_G(v)}x_e=b_v$
is replaced by $\sum_{e\in\delta_G(v)}x_e\le b_v$ for certain vertices $v$ of $G$.
\item Polytopes defined by modifying~\eqref{PGb} so that additional conditions $x_e\le c_e$
apply to certain edges $e$ of $G$,
where $c_e$ is a prescribed nonnegative number.
\item Polytopes of $b$-flows (or $b$-transshipments) on directed graphs.
See, for example,
Schrijver~\cite[Secs.~11.4 \&~13.2c]{Sch03}, or Korte and Vygen~\cite[Sec.~9.1]{KorVyg12}.
\item Certain other matching-type polytopes, including some of those discussed in Section~\ref{relsect}.
\item Polytopes of magic labelings of graphs.  See, for example, Ahmed~\cite{Ahm08}.
\item The symmetric transportation polytope $\N(R)$, and the related polytopes $\N(\le R)$,
$\N_{\le Z}(R)$ and $\N_{\le Z}(\le R)$.  See, for example, Brualdi~\cite[Sec.~8.2]{Bru06} for
definitions of the notation, and further information.
The cases $\N(R)$ and $\N_{\le Z}(R)$
will also be discussed briefly in Section~\ref{mat}.
\item The transportation polytope $\N(R,S)$, and the related polytopes
$\N(\le R,\le S)$, $\N_{\le Z}(R,$ $S)$ and $\N_{\le Z}(\le R,\le S)$.
See, for example, Brualdi~\cite[Secs.~8.1 \& 8.4]{Bru06} for
definitions of the notation, and further information.
The cases $\N(R,S)$ and $\N_{\le Z}(R,S)$ will also be discussed briefly in Section~\ref{mat}.
\item The polytope of doubly stochastic matrices, also known as the Birkhoff or assignment polytope,
and various related polytopes, including the polytopes of doubly substochastic matrices,
extensions of doubly substochastic matrices,
symmetric doubly stochastic matrices,
symmetric doubly substochastic matrices, and tridiagonal doubly stochastic matrices.
See, for example, Brualdi~\cite[Ch.~9]{Bru06}, and (for the tridiagonal case) Dahl~\cite{Dah04}.
\item The alternating sign matrix polytope. See, for example, Behrend and Knight~\cite[Sec.~6]{BehKni07}, or Striker~\cite{Str09}.
\end{list}

In the third paper~\cite{Beh13c}, the polytope of all elements of $\P(G,b)$ (for certain cases of $G$ and $b$)
which are invariant under a certain natural action of all elements of a
group of automorphisms of $G$ will be considered.

\section{Matrix forms of $\mathcal{P}(G,b)$}\label{mat}
In this section, some alternative forms involving matrices are given for~$\P(G,b)$.
The form~\eqref{IMF} will be used first in Section~\ref{facesPGb}, while the other forms,~\eqref{AMF} and~\eqref{BAMF}, will be used
in~\cite{Beh13b}. Certain matrix classes, and their relationship with certain cases of
fractional perfect $b$-matching polytopes, are also discussed.

The incidence matrix $I_G$ of a graph~$G$ is the matrix with rows and columns indexed by~$V$ and~$E$ respectively,
and entries $(I_G)_{ve}$ given by $1$ or $0$ according to whether or not
vertex~$v$ is incident with edge~$e$.  It follows immediately from this definition, and the definition~\eqref{PGb}
of the fractional perfect $b$-matching polytope of~$G$, that
\begin{equation}\label{IMF}\P(G,b)=\bigl\{x\in\R^E\bigm|x_e\ge0\mbox{ for each }e\in E,\ I_G\,x=b\bigr\}.\end{equation}

For $x\in\R^E$, define the generalized adjacency matrix $A_G(x)$ of $G$ to be the matrix with rows and columns indexed by~$V$, and
entries given by
\begin{equation}A_G(x)_{vw}=\sum_{e\in\delta_G(v)\cap\delta_G(w)}\!\!x_e\end{equation}
for each $v,w\in V$.
Note that the sum here is simply over all edges which connect $v$ and~$w$, that $A_G(x)$ is symmetric, and that
if $x_e=1$ for all $e\in E$ then $A_G(x)$ is the standard adjacency matrix of $G$.
It follows that
\begin{equation}\label{AMF}\P(G,b)=\biggl\{x\in\R^E\biggm|x_e\ge0\mbox{ for each }e\in E, \
\sum_{w\in V}A_G(x)_{vw}=b_v\mbox{ for each }v\in V\biggr\},\end{equation}
i.e., $\P(G,b)$ is the polytope of all assignments $x$ of nonnegative real numbers to the
edges of~$G$ such that the sum of entries in row/column $v$ of $A_G(x)$ is~$b_v$, for each vertex~$v$.

For the case in which $G$ is bipartite with bipartition $(U,W)$, and for $x\in\R^E$, define
the generalized $(U,W)$-biadjacency matrix $A^{(U,W)}_G\!(x)$ of $G$ to be
the submatrix of $A_G(x)$ obtained by restricting the rows to those indexed by~$U$ and
the columns to those indexed by~$W$.
If $x_e=1$ for all $e\in E$, then $A^{(U,W)}_G\!(x)$ is the standard $(U,W)$-biadjacency matrix of $G$.
It follows that
\begin{multline}\label{BAMF}\P(G,b)=\biggl\{x\in\R^E\biggm|
\displaystyle x_e\ge0\mbox{ for each }e\in E, \ \sum_{w\in V}A^{(U,W)}_G\!(x)_{uw}=b_u\mbox{ for each }u\in U,\\[-2mm]
\displaystyle\sum_{u\in U}A^{(U,W)}_G\!(x)_{uw}=b_w\mbox{ for each }w\in W\biggr\},\end{multline}
i.e., $\P(G,b)$ is the polytope of all assignments $x$ of nonnegative real numbers to the
edges of~$G$ such that the sum of entries in row~$u$ of $A^{(U,W)}_G\!(x)$ is~$b_u$
and the sum of entries in column~$w$ of $A^{(U,W)}_G\!(x)$ is~$b_w$, for each $u\in U$, $w\in W$.

The relationship between certain cases of fractional perfect $b$-matching polytopes,
and certain matrix classes will now be considered briefly.  Further details for these cases will be given in~\cite{Beh13b}.

It follows from~\eqref{AMF} that if $G$ does not contain multiple edges, then $\P(G,b)$ is affinely isomorphic (using
simple and obvious mappings) to the
polytope of all $|V|\times|V|$ symmetric matrices with nonnegative real entries, for which
certain entries are prescribed to be zero, and the sum of entries in any row/column is a prescribed nonnegative real number for that row/column.
An account of such polytopes is given by Brualdi in~\cite[Sec.~8.2]{Bru06}.  The notation used there is that,
given a vector $R\in\R^n$ with nonnegative entries, and a symmetric $n\times n$ matrix $Z$ each of whose entries is $0$ or~$1$,
$\N_{\le Z}(R)$ is the polytope, or matrix class, of all $n\times n$ symmetric matrices with nonnegative real entries,
for which the $i,j$ entry is zero if~$Z_{ij}$ is zero, and the sum of entries in row/column $i$ is~$R_i$.
Accordingly, if~$G$ does not contain multiple edges, then $\P(G,b)$ and $\N_{\le Z}(R)$ are affinely isomorphic, where~$Z$ is the adjacency matrix
of $G$, and associated entries of $b$ and~$R$ are equal.
For the case in which $G$ is a complete graph with loops (i.e., a graph in which any two distinct vertices are connected by a single edge,
and a single loop is incident
to each vertex),~$Z$ is a square matrix each of whose entries is 1, and $\N_{\le Z}(R)$ is
a so-called symmetric transportation polytope, denoted in~\cite[pp.~39~\&~348]{Bru06} as $\N(R)$.
It follows (using the fact that setting a subset of a polyhedron's
defining inequalities to equalities gives a, possibly empty, face of the polyhedron) that $\N_{\le Z}(R)$ is a face of $\N(R)$.
Similarly, for an arbitrary graph $G$ without multiple edges, $\P(G,b)$ is affinely isomorphic to a face of $\P(K_{V},b)$, where $K_{V}$ is
a complete graph with loops, and vertex set $V$.

It follows from~\eqref{BAMF} that if $G$ is bipartite with bipartition $(U,W)$,
and does not contain multiple edges, then $\P(G,b)$ is affinely isomorphic (again, using
simple and obvious mappings) to the
polytope of all $|U|\times|W|$ matrices with nonnegative real entries, for which
certain entries are prescribed to be zero, and the sum of entries in any row or
column is a prescribed nonnegative real number for that row or column.
An account of such polytopes is given by Brualdi in~\cite[Secs.~8.1 \&~8.4]{Bru06}.
The notation used there is that,
given vectors $R\in\R^m$ and $S\in\R^n$ with nonnegative entries, and an $m\times n$ matrix $Z$ each of whose entries is $0$ or~$1$,
$\N_{\le Z}(R,S)$ is the polytope, or matrix class, of all $m\times n$ matrices with nonnegative real entries,
for which the $i,j$ entry is zero if~$Z_{ij}$ is zero, the sum of entries in row $i$ is $R_i$, and the sum of entries in column $j$ is $S_j$.
Accordingly, if~$G$ is bipartite with bipartition $(U,W)$, and does not contain multiple edges, then
$\P(G,b)$ and $\N_{\le Z}(R,S)$ are affinely isomorphic, where~$Z$ is the $(U,W)$-biadjacency matrix
of $G$, the entries~$b_v$ with $v\in U$ are equal to associated entries of $R$,
and the entries~$b_v$ with $v\in W$ are equal to associated entries of $S$.
For the case in which $G$ is a complete bipartite graph (i.e., a graph in which each vertex of $U$ and each vertex of $W$ are connected by a
single edge),~$Z$ is a matrix each of whose entries is 1, and $\N_{\le Z}(R,S)$ is a so-called transportation polytope,
denoted in~\cite[pp.~26~\&~337]{Bru06} as $\N(R,S)$.
It follows that $\N_{\le Z}(R,S)$ is a face of $\N(R,S)$.
Similarly, for an arbitrary bipartite graph $G$ without multiple edges, and with bipartition $(U,W)$,
$\P(G,b)$ is affinely isomorphic to a face of $\P(K_{(U,W)},b)$, where $K_{(U,W)}$ is
a complete bipartite graph with bipartition $(U,W)$.  For further
information regarding transportation polytopes, see, for example,
Kim~\cite{Kim10}, Klee and Witzgall~\cite{KleWit68}, Schrijver~\cite[Sec.~21.6]{Sch03},
or Yemelichev, Kovalev and Kravtsov~\cite[Ch.~6]{YemKovKra84}.

\section{Conditions for nonemptiness and strictly positive elements of~$\mathcal{P}(G,b)$}\label{nonempt}
In this section, results which provide necessary and sufficient conditions for $\P(G,b)$ to be nonempty, or to contain a strictly positive element,
are obtained.
The conditions for~$\P(G,b)$ to be nonempty take the form of
finitely-many weak inequalities and equalities for certain sums of entries of $b$,
and the conditions for $\P(G,b)$ to contain a strictly positive element take the form of
finitely-many strict inequalities and equalities for certain sums of entries of $b$.
The conditions for~$\P(G,b)$ to contain a strictly positive element will be used in Section~\ref{furthcond}.

It will be simplest, in this section, to obtain results first for the case of bipartite~$G$ (in Theorems~\ref{bipnonemptth} and~\ref{bipposth}),
and then to use these to obtain
results for the case of arbitrary~$G$ (in Theorems~\ref{nonemptth} and~\ref{posth}).  By contrast, in later sections of the paper,
results will be obtained first for arbitrary $G$, with results for bipartite $G$ then following as corollaries.

The notation $\cupdot$ and $G[U,W]$ (for subsets $U$ and $W$ of $V$), as introduced in Section~\ref{convnot}, will be used in
this section.  Also, several of the results will be expressed in terms of vertex covers of $G$.
However, such results could easily be restated in terms of stable (or independent) sets of~$G$
(i.e., subsets of~$V$ which do not contain any adjacent vertices),
since $S\subset V$ is a stable set of $G$ if and only if $V\setminus S$ is a vertex cover of $G$.

In the first proposition of this section, it is seen that there would
be no loss of generality in the main theorems of this section
if $G$ were assumed to contain only single edges.  The terminology used in this proposition is
that a graph $G'$ is related to $G$ by reduction of multiple edges to single edges
if~$G'$ has vertex set $V$, $G'$ does not contain multiple edges, and, for any $u,w\in V$,~$u$
and~$w$ are adjacent in $G'$ if and only if $u$ and $w$ are adjacent in~$G$.

\begin{proposition}\label{multempt}
Let $G'$ be a graph related to $G$ by reduction of multiple edges to single edges.  Then
$\P(G,b)$ is nonempty if and only if $\P(G',b)$ is nonempty, and
$\P(G,b)$ contains a strictly positive element if and only if $\P(G',b)$ contains a strictly positive element.
\end{proposition}
This elementary result will now be proved directly.  However, its validity will also follow from later theorems which give
conditions for $\P(G,b)$ to be nonempty, or to contain a strictly positive element, since it will be apparent that
these conditions depend only on whether or not certain pairs of vertices of $G$ are adjacent, rather the actual number of edges
which connect such vertices.
\begin{proof}Denote the edge set of $G'$ as $E'$.
First, let $\P(G,b)$ be nonempty and choose an $x$ in $\P(G,b)$.
Define $x'\in\R^{E'}$ by $x'_{e'}=\sum_{e\in G[\{u\},\{w\}]}x_e$ for each $e'\in E'$, where $u$ and $w$ are the endpoints of $e'$.
Then $x'\in\P(G',b)$, and $x'$ is strictly positive if~$x$ is strictly positive.

Conversely, let $\P(G',b)$ be nonempty and choose an $x'$ in $\P(G',b)$.  Define
$x\in\R^{E}$ by $x_e=x'_{e'}/|G[\{u\},\{w\}]|$ for each $e\in E$, where
$u$ and $w$ are the endpoints of $e$, and $e'$ is the single edge of $E'$
which connects $u$ and $w$.
Then $x\in\P(G,b)$, and $x$ is strictly positive if~$x'$ is strictly positive.
\end{proof}

The following elementary result provides necessary conditions for~$\P(G,b)$ to be non\-empty,
or to contain a strictly positive element.
\begin{lemma}\label{nonempnec}\mbox{}\\[-5mm]
\begin{list}{(\roman{listnumber})}{\usecounter{listnumber}\setlength{\labelwidth}{6mm}\setlength{\leftmargin}{9mm}\setlength{\itemsep}{1mm}}
\item A necessary condition for $\P(G,b)$ to be nonempty is that
$\sum_{v\in C}b_v\ge\sum_{v\in V\setminus C}b_v$ for each vertex cover~$C$ of~$G$.
\item A necessary condition for $\P(G,b)$ to contain a strictly positive element
is that the condition of (i) is satisfied, with its inequality holding as an equality if and only if $V\setminus C$ is also a vertex cover of~$G$.
\end{list}
\end{lemma}
Note that, by adding $\sum_{v\in C}b_v$ to each side,
the inequality in (i) of this lemma is equivalent to
$\sum_{v\in C}b_v\ge\frac{1}{2}\sum_{v\in V}b_v$.

Note also that, for a subset $C$ of $V$, $C$ and $V\setminus C$ are both vertex covers of $G$ if and only if $(C,V\setminus C)$ is a bipartition for $G$.
Hence, if $G$ is not bipartite, then a necessary condition for $\P(G,b)$ to contain a strictly positive element
is that
$\sum_{v\in C}b_v>\sum_{v\in V\setminus C}b_v$ for each vertex cover $C$ of~$G$.
\begin{proof}
Assume that $\P(G,b)$ is nonempty, and choose an $x\in\P(G,b)$.
Then $\sum_{e\in\delta_G(v)}x_e=b_v$ for each $v\in V$, which gives
$\sum_{v\in C}b_v=\sum_{e\in G[C,C]}\mu_e\,x_e+\sum_{e\in G[C,V\setminus C]}x_e$, for any $C\subset V$,
where $\mu_e=2$ if $e$ is not a loop and $\mu_e=1$ if $e$ is a loop.  Therefore,
\begin{equation}\label{Ceq}\textstyle\sum_{v\in C}b_v-\sum_{v\in V\setminus C}b_v=\sum_{e\in G[C,C]}\mu_e\,x_e-
\sum_{e\in G[V\setminus C,V\setminus C]}\mu_e\,x_e,\end{equation}
for any $C\subset V$.

Part (i) of the lemma
now follows from~\eqref{Ceq}, and the facts that $\mu_e\,x_e\ge0$ for each $e\in E$,
 and $G[V\setminus C,V\setminus C]=\emptyset$
if~$C$ is a vertex cover. Part (ii) of the lemma follows from~\eqref{Ceq}
by assuming that $\P(G,b)$ contains a strictly positive element, choosing~$x$ to be such an element,
and using the facts that $G[C,C]\ne\emptyset=G[V\setminus C,V\setminus C]$ if~$C$ is
a vertex cover and $V\setminus C$ is not a vertex cover, while
$G[C,C]=G[V\setminus C,V\setminus C]=\emptyset$ if~$C$ and $V\setminus C$ are both vertex covers.
\end{proof}

The next two results, Theorems~\ref{bipnonemptth} and~\ref{bipposth}, state that if $G$ is bipartite, then the conditions of Lemma~\ref{nonempnec} are
also sufficient to ensure that $\P(G,b)$ is nonempty, or (for $E\ne\emptyset$) that~$\P(G,b)$ contains a strictly positive element.
\begin{theorem}\label{bipnonemptth}
Let $G$ be bipartite.
Then a necessary and sufficient condition for $\P(G,b)$ to be nonempty is that
$\sum_{v\in C}b_v\ge\sum_{v\in V\setminus C}b_v$ for each vertex cover $C$ of~$G$.
\end{theorem}
It can be seen that, if $(U,W)$ is a bipartition for $G$, then
the condition of the theorem is equivalent to the alternative condition that
$\sum_{v\in U_1}b_v+\sum_{v\in W_1}b_v\ge\sum_{v\in U_2}b_v+\sum_{v\in W_2}b_v$ for all sets~$U_1$,~$U_2$,~$W_1$ and~$W_2$
such that $U=U_1\cupdot U_2$, $W=W_1\cupdot W_2$ and $G[U_2,W_2]=\emptyset$.
(In particular, for $C$ satisfying the condition of the theorem, set $U_1=U\cap C$, $U_2=U\setminus C$, $W_1=W\cap C$ and $W_2=W\setminus C$,
and conversely, for $U_1$, $U_2$, $W_1$ and~$W_2$ satisfying the alternative condition, set $C=U_1\cup W_1$.)
It can also be checked that the alternative condition
remains unchanged if its single inequality is replaced by
$\sum_{v\in U_1}b_v\ge\sum_{v\in W_2}b_v$ and $\sum_{v\in W_1}b_v\ge\sum_{v\in U_2}b_v$,
by $\sum_{v\in U}b_v=\sum_{v\in W}b_v$ and $\sum_{v\in U_1}b_v\ge\sum_{v\in W_2}b_v$,
or by $\sum_{v\in U}b_v=\sum_{v\in W}b_v$ and $\sum_{v\in W_1}b_v\ge\sum_{v\in U_2}b_v$.
(For example, the condition $\sum_{v\in U}b_v=\sum_{v\in W}b_v$ follows from the condition
of the theorem by using the vertex covers $C=U$ and $C=W$.)

Note that, in the condition of this theorem,
$C$ can be restricted to being different from $V$ (since taking $C$ to be $V$ gives
a vertex cover of $G$ which automatically satisfies $\sum_{v\in C}b_v\ge\sum_{v\in V\setminus C}b_v$).

It can also be seen that, if the condition of this theorem is satisfied, and if $C$ and $V\setminus C$ are both vertex covers
of $G$, then $\sum_{v\in C}b_v=\sum_{v\in V\setminus C}b_v$ (i.e., the inequality then holds as an equality).

This theorem is a standard result.  See, for example, Schrijver~\cite[Thm.~21.11]{Sch03}.
It can be proved using linear programming duality (as done in the proof given by Schrijver~\cite[Thm.~21.11]{Sch03}),
or using standard theorems from network flow theory. (For example, it follows from Schrijver~\cite[Cor.~11.2h]{Sch03} by
using a directed graph which is formed from $G$ by directing each edge from $U$ to $W$, where
$(U,W)$ is a bipartition for $G$).
For completeness, a proof will also be given here.  This is a direct and self-contained proof, which uses an approach based on that
used by Schrijver~\cite[Thms.~10.3 and~11.2]{Sch03} for proofs of
the max-flow min-cut theorem and Hoffman's circulation theorem.
\begin{proof}The necessity of the condition is given by (i) of Lemma~\ref{nonempnec}.

The sufficiency of the condition will be obtained
by using a bipartition $(U,W)$ for~$G$, and showing that if $\P(G,b)$ is empty and $\sum_{v\in U}b_v=\sum_{v\in W}b_v$, then
there exist~$U_1$,~$U_2$,~$W_1$ and~$W_2$ such that $U=U_1\cupdot U_2$, $W=W_1\cupdot W_2$, $G[U_2,W_2]=\emptyset$
and $\sum_{v\in W_1}b_v<\sum_{v\in U_2}b_v$.

So, let $\P(G,b)=\emptyset$ and $\sum_{v\in U}b_v=\sum_{v\in W}b_v$.
Define $\R^E_+=\{x\in\R^E\mid x_e\ge0$ for each $e\in E\}$ and,
for any $x\in\R^E_+$, let $f(x)=\sum_{v\in V}|\sum_{e\in\delta_G(v)}x_e-b_v|$.
(Note that $f(x)>0$ for all $x\in\R^E_+$, since $\P(G,b)=\emptyset$.)
Now choose an~$x$ which minimizes~$f$ over~$\R^E_+$.  The forms of
$\R^E_+$ and~$f$ guarantee the existence of such an $x$.  (In particular,
the polyhedron~$\R^E_+$ can be subdivided into finitely-many nonempty polyhedra, on each of which $f$ is a
positive affine function.  Specifically, each such polyhedron $P$ has the form
$\{x\in\R^E_+\mid\sigma(P)_v\,\bigl(\sum_{e\in\delta_G(v)}x_e-b_v\bigr)\ge0\mbox{ for each }v\in V\}$,
for some assignment $\sigma(P)_v$ of~$-1$ or $1$ to each $v\in V$,
so that $f(x)=\sum_{v\in V}\sigma(P)_v\,\bigl(\sum_{e\in\delta_G(v)}x_e-b_v\bigr)$ for all $x\in P$.
The standard fact, as given for example in
Korte and Vygen~\cite[Prop.~3.1]{KorVyg12}, that a real affine function which is bounded below on a nonempty polyhedron
attains a minimum over the polyhedron then implies that $f$ attains a minimum over~$\R^E_+$.)

Define $S=\{u\in U\mid\sum_{e\in\delta_G(u)}x_e<b_u\}\cup\{w\in W\mid\sum_{e\in\delta_G(w)}x_e>b_w\}$
and $T=\{u\in U\mid\sum_{e\in\delta_G(u)}x_e>b_u\}\cup\{w\in W\mid\sum_{e\in\delta_G(w)}x_e<b_w\}$.
Since $\P(G,b)=\emptyset$, $S\cup T$ is nonempty. It then follows, using $\sum_{v\in U}b_v=\sum_{v\in W}b_v$, that~$S$
and~$T$ are each nonempty (since $S\ne\emptyset$ and $T=\emptyset$ would give $\sum_{v\in U}b_v>\sum_{v\in W}b_v$, while
$S=\emptyset$ and $T\ne\emptyset$ would give $\sum_{v\in U}b_v<\sum_{v\in W}b_v$).
Now define
\begin{multline}\label{SS}S'=\{v\in V\mid\text{there exists $s\in S$ and a path $P$ in $G$ from~$s$ to~$v$ satisfying}\\
x_e>0\text{ for each edge $e$ corresponding to a step of $P$ from $W$ to }U\},\end{multline}
i.e., $S'$ is the set of vertices of $G$ which are reachable from $S$ by a path $P$
with the property that $x_e$ is positive for each edge $e$ which corresponds to a step of $P$ from $W$ to $U$.
It follows immediately that $S\subset S'$, $G[U\cap S',W\setminus S']=\emptyset$, and
$x_e=0$ for each $e\in G[W\cap S',U\setminus S']$.
Also, $S'\cap T=\emptyset$, where this can be deduced as follows.  If $S'\cap T$ were nonempty,
then there would exist $s\in S$, $t\in T$ and a path~$P$ from~$s$ to~$t$ satisfying the property of~\eqref{SS}.
Taking $y\in\R^E$ as $y_e=\epsilon$ for each edge $e$ corresponding to a step of~$P$ from~$U$ to~$W$,
$y_e=-\epsilon$ for each edge $e$ corresponding to a step of $P$ from~$W$ to $U$, and $y_e=0$ for each edge $e$ not in $P$, it would follow
that, for sufficiently small $\epsilon>0$, $x+y\in\R^E_+$ and $f(x+y)<f(x)$, but this is impossible
since $x$ minimizes $f$ over~$\R^E_+$.

Now define $U_1=U\setminus S'$, $U_2=U\cap S'$, $W_1=W\cap S'$ and $W_2=W\setminus S'$.
Then $U=U_1\cupdot U_2$, $W=W_1\cupdot W_2$, $G[U_2,W_2]=\emptyset$,
and $x_e=0$ for each $e\in G[U_1,W_1]$. Also, $\sum_{e\in\delta_G(u)}x_e\le b_u$ for each $u\in U_2$,
and $\sum_{e\in\delta_G(w)}x_e\ge b_w$ for each $w\in W_1$
(since $S\subset S'=U_2\cup W_1$ and $T\subset V\setminus S'$),
with strict inequality holding for at least one $u\in U_2$ or $w\in W_1$ (since $S\ne\emptyset$).
Therefore $\sum_{w\in W_1}b_w\le\sum_{e\in G[U,W_1]}x_e=\sum_{e\in G[U_2,W_1]}x_e=
\sum_{e\in G[U_2,W]}x_e\le\sum_{u\in U_2}b_u$, with
at least one of the inequalities holding strictly,
so that $\sum_{v\in W_1}b_v<\sum_{v\in U_2}b_v$, as required.\end{proof}

\begin{theorem}\label{bipposth}
Let $G$ be bipartite, with $E$ nonempty.
Then a necessary and sufficient condition for $\P(G,b)$ to contain a strictly positive element is that
$\sum_{v\in C}b_v\ge\sum_{v\in V\setminus C}b_v$ for each vertex cover~$C$ of $G$
(i.e., the condition of Theorem~\ref{bipnonemptth} is satisfied),
with the inequality holding as an equality if and only if $V\setminus C$ is also a vertex cover of~$G$.
\end{theorem}
Note that if $(U,W)$ is a bipartition for $G$, then
the condition of the theorem is equivalent to the condition that
$\sum_{v\in U_1}b_v+\sum_{v\in W_1}b_v\ge\sum_{v\in U_2}b_v+\sum_{v\in W_2}b_v$ for all sets~$U_1$,~$U_2$,~$W_1$ and~$W_2$
such that $U=U_1\cupdot U_2$, $W=W_1\cupdot W_2$ and $G[U_2,W_2]=\emptyset$, with the inequality holding
as an equality if and only if $G[U_1,W_1]=\emptyset$.
Furthermore, this condition remains unchanged if its inequality is replaced by
$\sum_{v\in U_1}b_v\ge\sum_{v\in W_2}b_v$, or by $\sum_{v\in W_1}b_v\ge\sum_{v\in U_2}b_v$
(since if the condition, in any of these forms, is satisfied, then taking
$U_1=U$, $W_2=W$ and $U_2=W_1=\emptyset$, or $U_2=U$, $W_1=W$ and $U_1=W_2=\emptyset$,
gives $\sum_{v\in U}b_v=\sum_{v\in W}b_v$).

This theorem, stated in terms of matrices, is due to Brualdi. See~\cite[Thm.~2.1]{Bru68},
\cite[Thm.~2.7]{Bru76} and~\cite[Thm.~8.1.7]{Bru06}.  The statement given by Brualdi can be
translated to that given here using the correspondence, discussed in Section~\ref{mat},
between $\P(G,b)$ and $\N_{\le Z}(R,S)$
for the case in which $G$ is bipartite and does not contain multiple edges, and Proposition~\ref{multempt}.
\begin{proof}The necessity of the condition is given by~(ii) of Lemma~\ref{nonempnec}.

Proceeding to the proof of sufficiency, define $d\in\R^V$ by
$d_v=|\delta_G(v)|$ for each $v\in V$ (i.e., $d_v$ is the degree of~$v$), and define $y\in\R^E$
(where $E\ne\emptyset$ ensures that $\R^E\ne\{0\}$)
by $y_e=1$ for each $e\in E$.
Then~$y$ is a strictly positive element of $\P(G,d)$.  Therefore,
using~(ii) of Lemma~\ref{nonempnec},
$\sum_{v\in C}d_v\ge\sum_{v\in V\setminus C}d_v$ for each vertex cover~$C$ of $G$,
with equality holding if and only if $V\setminus C$ is also a vertex cover of~$G$.

Now assume that the condition of the theorem is satisfied.
It can then be shown that $b_v>0$ for each $v\in V$ with $d_v>0$, i.e.,
for each nonisolated vertex $v$.
(More specifically, this can be done by considering a nonisolated vertex $w$,
and a bipartition $(U,W)$ for $G$, with $w\in W$.
Then, choosing the vertex cover $C=U$ gives $\sum_{v\in U}b_v=\sum_{v\in W}b_v$,
while choosing the vertex cover $C=U\cup\{w\}$ gives $\sum_{v\in U\cup\{w\}}b_v>\sum_{v\in W\setminus\{w\}}b_v$,
from which it follows that $b_w>0$.)

Now choose an $\epsilon>0$ which satisfies $\epsilon\,d_v\le b_v$ for each $v\in V$, and
$\epsilon\,(\sum_{v\in C}d_v-\sum_{v\in V\setminus C}d_v)\le\sum_{v\in C}b_v-\sum_{v\in V\setminus C}b_v$ for each
vertex cover $C$ of $G$,
where the conditions satisfied by $b$ and $d$ guarantee the existence of such an $\epsilon$.
It follows that $b-\epsilon\,d$ has all of its entries nonnegative, and satisfies the condition of Theorem~\ref{bipnonemptth}
(i.e., $\sum_{v\in C}(b_v-\epsilon\,d_v)\ge\sum_{v\in V\setminus C}(b_v-\epsilon\,d_v)$ for each vertex cover~$C$ of $G$),
so that $\P(G,b-\epsilon\,d)\ne\emptyset$.  Finally, choose an $x\in\P(G,b-\epsilon\,d)$.  Then it can be seen
that $x+\epsilon\,y$ is a strictly positive element of~$\P(G,b)$.
\end{proof}

Theorems~\ref{bipnonemptth} and~\ref{bipposth}, which apply to the case of bipartite $G$, can now be used to give
analogous results for the case of arbitrary $G$. These results will be expressed in preliminary forms in Lemma~\ref{nonemptlem},
and then restated in more compact forms in Theorems~\ref{nonemptth} and~\ref{posth}.
\newpage
\begin{lemma}\label{nonemptlem}
\mbox{}\\[-5mm]
\begin{list}{(\roman{listnumber})}{\usecounter{listnumber}\setlength{\labelwidth}{6mm}\setlength{\leftmargin}{9mm}\setlength{\itemsep}{1mm}}
\item A necessary and sufficient condition for $\P(G,b)$ to be nonempty is that $\sum_{v\in U_1}b_v\ge\sum_{v\in W_2}b_v$
for all sets~$U_1$,~$U_2$,~$W_1$ and~$W_2$ such that $V=U_1\cupdot U_2=W_1\cupdot W_2$ and $G[U_2,W_2]=\emptyset$.
\item Let $E$ be nonempty. Then a necessary and sufficient condition for $\P(G,b)$ to contain a strictly positive element is that
the condition of (i) is satisfied, with its inequality holding as an equality if and only if $G[U_1,W_1]=\emptyset$.
\end{list}
\end{lemma}
Note that the conditions in the lemma remain unchanged if the inequality is replaced by $\sum_{v\in W_1}b_v\ge\sum_{v\in U_2}b_v$.
\begin{proof}
Let $G'$ be a so-called `bipartite double graph' of $G$.  Specifically, let $G'$ have vertex set $V'=V\times\{1,2\}$
and edge set $E'=E\times\{1,2\}$, where if $e\in E$ connects vertices $u$ and~$w$ of~$V$, then
one of the edges $(e,1)$ or $(e,2)$ of $E'$ connects vertices $(u,1)$ and $(w,2)$ of $V'$,
while the other connects vertices $(w,1)$ and $(u,2)$ of $V'$.
Also define $b'\in\R^{V'}$ by $b'_{(v,1)}=b'_{(v,2)}=b_v$ for each $v\in V$.
It can now be checked that $\P(G,b)\ne\emptyset$ if and only if $\P(G',b')\ne\emptyset$.
In particular, if there exists $x\in\P(G,b)$, then there exists $x'\in\P(G',b')$ given by $x'_{(e,1)}=x'_{(e,2)}=\mu_ex_e/2$ for
each $e\in E$, where $\mu_e=2$ if $e$ is not a loop and $\mu_e=1$ if $e$ is a loop.
Conversely, if there exists $x'\in\P(G',b')$, then there exists $x\in\P(G,b)$ given by $x_e=(x'_{(e,1)}+x'_{(e,2)})/\mu_e$ for
each $e\in E$.
It follows similarly that $\P(G,b)$ contains a strictly positive element if and only if $\P(G',b')$ contains a strictly positive element.

Since $G'$ is bipartite, with bipartition $(V\times\{1\},V\times\{2\})$,
it follows from Theorem~\ref{bipnonemptth} (using one of the alternative forms given after the statement of that theorem,
and noting that $\sum_{v\in V\times\{1\}}b'_v=\sum_{v\in V\times\{2\}}b'_v$),
that a necessary and sufficient condition for $\P(G,b)$ to be nonempty is that
$\sum_{v\in U'_1}b'_v\ge\sum_{v\in W'_2}b'_v$ for all sets~$U'_1$,~$U'_2$,~$W'_1$ and~$W'_2$
such that $V\times\{1\}=U'_1\cupdot U'_2$, $V\times\{2\}=W'_1\cupdot W'_2$ and $G'[U'_2,W'_2]=\emptyset$.

Similarly, it follows from Theorem~\ref{bipposth} (using one of the alternative forms given after the statement of that theorem),
that a necessary and sufficient condition for $\P(G,b)$ to contain a strictly positive element is that
the previous condition for nonemptiness is satisfied,
with its inequality holding as an equality if and only if $G'[U'_1,W'_1]=\emptyset$.

Finally, it can easily be seen that the previous two conditions are equivalent to the
corresponding conditions of the lemma.
\end{proof}

\begin{theorem}\label{nonemptth}
A necessary and sufficient condition for $\P(G,b)$ to be nonempty is that
$\sum_{v\in V_1}b_v\ge\sum_{v\in V_3}b_v$ for all
sets $V_1$, $V_2$ and $V_3$ such that $V=V_1\cupdot V_2\cupdot V_3$ and $G[V_2\cup V_3,V_3]=\emptyset$.
\end{theorem}
Note that the appearance of $V_2$ in this theorem could be removed by
rewriting the condition as $\sum_{v\in V_1}b_v\ge\sum_{v\in V_3}b_v$ for all
disjoint subsets $V_1$ and $V_3$ of $V$ such that $G[V\setminus V_1,V_3]=\emptyset$.
Note also that $G[V_2\cup V_3,V_3]=\emptyset$ is equivalent to $G[V_2,V_3]=G[V_3,V_3]=\emptyset$.
Furthermore, in the condition of the theorem, $V_3$ can be restricted to being nonempty
(since if sets $V_1$, $V_2$ and $V_3$ satisfy $V=V_1\cupdot V_2\cupdot V_3$ and $V_3=\emptyset$,
then $G[V_2\cup V_3,V_3]=\emptyset$ and $\sum_{v\in V_1}b_v\ge\sum_{v\in V_3}b_v$ are automatically
satisfied).

It can also be seen that, if the condition of this theorem is satisfied, and if
sets $V_1$,~$V_2$ and~$V_3$ satisfy $V=V_1\cupdot V_2\cupdot V_3$ and $G[V_1,V_1\cup V_2]=G[V_2\cup V_3,V_3]=\emptyset$.
then $\sum_{v\in V_1}b_v=\sum_{v\in V_3}b_v$  (i.e., the inequality then holds as an equality).

\begin{proof}It will be shown that the condition of the theorem is equivalent to the condition of~(i) of Lemma~\ref{nonemptlem}.
(Alternatively, the necessity of the condition of the theorem could easily be proved directly.)

First, let the condition of (i) of Lemma~\ref{nonemptlem} be satisfied, and
consider any sets~$V_1$,~$V_2$ and~$V_3$ for which
$V=V_1\cupdot V_2\cupdot V_3$ and $G[V_2\cup V_3,V_3]=\emptyset$.
Now take~$U_1$, $U_2$,~$W_1$ and~$W_2$ to be
$U_1=V_1$, $U_2=V_2\cup V_3$, $W_1=V_1\cup V_2$ and $W_2=V_3$.
Then~$V=U_1\cupdot U_2=W_1\cupdot W_2$ and $G[U_2,W_2]=\emptyset$,
so that, since the condition of (i) of Lemma~\ref{nonemptlem} is satisfied,
$\sum_{v\in U_1}b_v\ge\sum_{v\in W_2}b_v$.  Therefore $\sum_{v\in V_1}b_v\ge\sum_{v\in V_3}b_v$,
and hence the condition of the theorem is satisfied.

Conversely, let the condition of the theorem be satisfied, and
consider any sets~$U_1$,~$U_2$, $W_1$ and~$W_2$ for which $V=U_1\cupdot U_2=W_1\cupdot W_2$ and $G[U_2,W_2]=\emptyset$.
Now take~$V_1$,~$V_2'$,~$V_2''$,~$V_2$ and~$V_3$ to be
$V_1=U_1\cap W_1$, $V_2'=U_2\cap W_1$, $V_2''=U_1\cap W_2$, $V_2=V_2'\cup V_2''$ and $V_3=U_2\cap W_2$.
Then~$U_1=V_1\cup V_2''$, $U_2=V_2'\cup V_3$, $W_1=V_1\cup V_2'$, $W_2=V_2''\cup V_3$, and
$V=V_1\cupdot V_2\cupdot V_3$.  Also, $\emptyset=G[U_2,W_2]=G[V_2'\cup V_3,V_2''\cup V_3]=G[V_2'\cup V_3,V_3]
\cup G[V_2''\cup V_3,V_3]\cup G[V_2',V_2'']=G[V_2\cup V_3,V_3]\cup G[V_2',V_2'']$, and therefore $G[V_2\cup V_3,V_3]=\emptyset$.
So, since the condition of the theorem is satisfied,
$\sum_{v\in V_1}b_v\ge\sum_{v\in V_3}b_v$, which gives $\sum_{v\in V_1\cup V_2''}b_v\ge\sum_{v\in V_3\cup V_2''}b_v$,
and thus $\sum_{v\in U_1}b_v\ge\sum_{v\in W_2}b_v$.
Hence, the condition of (i) of Lemma~\ref{nonemptlem} is satisfied.\end{proof}

\begin{theorem}\label{posth}
Let $E$ be nonempty.  Then a necessary and sufficient condition for $\P(G,b)$ to contain a strictly positive element
is that $\sum_{v\in V_1}b_v\ge\sum_{v\in V_3}b_v$ for all
sets $V_1$, $V_2$ and~$V_3$ such that $V=V_1\cupdot V_2\cupdot V_3$ and $G[V_2\cup V_3,V_3]=\emptyset$
(i.e., the condition of Theorem~\ref{nonemptth} is satisfied),
with the inequality holding as an equality if and only if $G[V_1,V_1\cup V_2]=\emptyset$.
\end{theorem}
This theorem, stated in terms of matrices, is due to Brualdi.
See~\cite[Thm.~3.7]{Bru76} and \cite[Thm.~8.2.3]{Bru06}.
The statement given by Brualdi can be
translated to that given here using the correspondence, discussed in Section~\ref{mat}, between $\P(G,b)$ and $\N_{\le Z}(R)$
for the case in which $G$ does not contain multiple edges, and Proposition~\ref{multempt}.
\begin{proof}It will be shown, by extending the proof
of Theorem~\ref{nonemptth}, that the condition of the theorem is equivalent to the condition of~(ii) of Lemma~\ref{nonemptlem}.
(Again, the necessity of the condition of the theorem could easily be proved directly instead.)

First, let the condition of (ii) of Lemma~\ref{nonemptlem} be satisfied,
consider any sets~$V_1$,~$V_2$ and~$V_3$ for which
$V=V_1\cupdot V_2\cupdot V_3$ and $G[V_2\cup V_3,V_3]=\emptyset$,
and take~$U_1$, $U_2$,~$W_1$ and~$W_2$ to be the same as in
the first part of the proof of Theorem~\ref{nonemptth}.
Then~$V=U_1\cupdot U_2=W_1\cupdot W_2$, $G[U_2,W_2]=G[V_2\cup V_3,V_3]=\emptyset$,
$G[U_1,W_1]=G[V_1,V_1\cup V_2]$, and $\sum_{v\in U_1}b_v-\sum_{v\in W_2}b_v=
\sum_{v\in V_1}b_v-\sum_{v\in V_3}b_v$.  It can now be seen that
the condition of the theorem is satisfied, since the condition of~(ii) of Lemma~\ref{nonemptlem} is satisfied.

Conversely, let the condition of the theorem be satisfied,
consider any sets~$U_1$,~$U_2$,~$W_1$ and~$W_2$ for which $V=U_1\cupdot U_2=W_1\cupdot W_2$ and $G[U_2,W_2]=\emptyset$,
and take~$V_1$,~$V_2'$,~$V_2''$,~$V_2$ and~$V_3$ to be the same as in the second part of the proof of Theorem~\ref{nonemptth}.
Then $V=V_1\cupdot V_2\cupdot V_3$,
and $\emptyset=G[U_2,W_2]=G[V_2\cup V_3,V_3]\cup G[V_2',V_2'']$, so that $G[V_2\cup V_3,V_3]=G[V_2',V_2'']=\emptyset$.
Also,  $G[U_1,W_1]=
G[V_1\cup V_2'',V_1\cup V_2']=
G[V_1,V_1\cup V_2']\cup G[V_1,V_1\cup V_2'']\cup G[V_2',V_2'']=G[V_1,V_1\cup V_2]$ (using $G[V_2',V_2'']=\emptyset$),
and $\sum_{v\in V_1}b_v-\sum_{v\in V_3}b_v=\sum_{v\in V_1\cup V_2''}b_v-\sum_{v\in V_3\cup V_2''}b_v=
\sum_{v\in U_1}b_v-\sum_{v\in W_2}b_v$.
It can now be seen that the condition of~(ii) of Lemma~\ref{nonemptlem} is satisfied, since the condition of the theorem is satisfied
\end{proof}

\section{Relevant results for graphs}\label{inc}
In this section, some relevant general results concerning the incidence matrix~$I_G$
of an arbitrary graph~$G$ (which may contain loops and multiple edges) are obtained.
Most of these results involve the nullity of~$I_G$ with respect to the field~$\R$,
i.e., the dimension of the kernel, or nullspace, of $I_G$ with respect to $\R$, where this kernel is
explicitly $\{x\in\R^E\mid I_G\,x=0\}=
\{x\in\R^E\mid\sum_{e\in\delta_G(v)}x_e=0\mbox{ for each }v\in V\}$.  The results will be applied to~$\P(G,b)$
in Section~\ref{facesPGb}.

Note that in the literature, the field $\{0,1\}$ is often used instead of $\R$, and in this
case the kernel of $I_G$ is the so-called cycle space of~$G$.
Note also that various results which are closely related to those of this section have appeared in the literature.
See, for example, Akbari, Ghareghani, Khosrovshahi and Maimani~\cite[Thm.~2]{AkbGhaKhoMai06}, or Villarreal~\cite[Cor.~3.2]{Vil95}.

\begin{proposition}\label{incnulllem}
The nullity of the incidence matrix of $G$ is $|E|-|V|+B$, where $B$ is the number of bipartite components of~$G$.\end{proposition}
Note that the fact that $\rank(A)=\rank(A^T)=n-\nullity(A)$, for any real matrix $A$ with~$n$ columns,
implies that the result of this proposition is equivalent to
\begin{equation}\label{incnulllem2}\rank(I_G)=|V|-B,\end{equation}
and to
\begin{equation}\label{incnulllem3}\nullity(I_G{}^T)=B,\end{equation}
where $B$ is again the number of bipartite components of~$G$.
These results, at least for the case of simple graphs, are
standard. (See, for example, Godsil and Royle~\cite[Thm.~8.2.1]{GodRoy04}.)

Note also that if $G$ is bipartite and planar, then it follows from this proposition, and Euler's formula for planar graphs
(which remains valid for graphs with multiple edges), that the nullity of the incidence matrix
of $G$ is the number of bounded faces in a planar embedding of~$G$.
\begin{proof}
The validity of the form~\eqref{incnulllem3} of the proposition will be confirmed.

The kernel of $I_G{}^T$ is $\{y\in\R^V\mid I_G{}^Ty=0\}=\{y\in\R^V\mid y_u=-y_w$ for all pairs~$u,w$ of adjacent vertices of~$G\}$.
By considering pairs of adjacent vertices successively along paths through
each component of~$G$, forming a bipartition $(U_C,W_C)$ for each bipartite component~$C$,
and using the fact that a nonbipartite component contains an odd-length cycle, it can be seen
that the general solution of the equations for $y$ is
\[y_v=\begin{cases}\lambda_C,&v\in U_C,\\
-\lambda_C,&v\in W_C,\\
0,&v\text{ is a vertex of a nonbipartite component,}\end{cases}\]
where $\lambda_C\in\R$ is arbitrary for each $C$.  It now follows that
$\nullity(I_G{}^T)=B$.\end{proof}

\begin{proposition}\label{zeronullity}
The nullity of the incidence matrix of $G$ is zero if and
only if each component of~$G$ either is acyclic or else contains exactly one cycle with that cycle having odd length.\end{proposition}
In this proposition, the choice of conditions for the components of $G$ applies independently to each component.
An alternative statement of the proposition is that the nullity of the incidence matrix of $G$ is zero
if and only if~$G$ has no even-length cycles and no component containing
more than one odd-length cycle.  It can also be seen that $\nullity(I_G)=0$ is equivalent to the condition that
$x=0$ is the only $x\in\R^E$ which satisfies
$\sum_{e\in\delta_G(v)}x_e=0$ for each $v\in V$.

Note that for the case of bipartite $G$, it follows from this proposition,
and the fact that a bipartite graph does not contain any odd-length cycles, that
the nullity of the incidence matrix of $G$ is zero if and
only if~$G$ is a forest.
\begin{proof}The kernel of $I_G$ is the direct sum of the kernels of the incidence matrices of its components.  Therefore,
$\nullity(I_G)=0$ if and only if $\nullity(I_C)=0$ for each component~$C$ of~$G$.  Applying Proposition~\ref{incnulllem}, these equations
are $|E_C|+1=|V_C|$ for each bipartite component~$C$ of $G$, and
$|E_C|=|V_C|$ for each nonbipartite component~$C$ of $G$, where $E_C$ and~$V_C$ are the edge
and vertex sets of $C$.  Using the fact that a connected graph~$C$ satisfies $|E_C|+1=|V_C|$ if and only if $C$ is acyclic,
and satisfies $|E_C|=|V_C|$ if and only if $C$ contains exactly one cycle, it now follows that $\nullity(I_G)=0$ if and only if
each bipartite component of $G$ is acyclic, and each nonbipartite component of $G$ contains exactly one cycle.  Finally, using the fact
that a graph is bipartite if and only if it does not contain any odd-length cycles, it follows that $\nullity(I_G)=0$ if and only if
each component of~$G$ either is acyclic or else contains exactly one cycle with that cycle having odd length.\end{proof}

Propositions~\ref{incnulllem} and \ref{zeronullity} can also be proved more directly.
Such alternative proofs provide further insight into these results,
so will now be outlined briefly.\\[2.5mm]
\emph{Alternative proof of Proposition~\ref{zeronullity}.}
First, let each component of~$G$ either be acyclic or else contain exactly one cycle with that cycle having odd length,
and let $x\in\R^E$ satisfy $I_G\,x=0$, i.e., $\sum_{e\in\delta_G(v)}x_e=0$ for each $v\in V$.
It follows immediately that $x_e=0$ for each pendant edge~$e$ (i.e., an edge incident to a univalent vertex).
By iteratively deleting such edges from~$E$ and considering the equation for $x$ at each univalent vertex $v$
in the resulting reduced graph, it then follows that $x_e=0$ for all edges $e$ of~$E$, except possibly those which are part of
disjoint cycles, where the length of each such cycle is odd and at least~3.
But if $e_1,\ldots,e_n$ are the edges of such a cycle, then the associated entries of~$x$ satisfy
$x_{e_n}+x_{e_1}=x_{e_1}+x_{e_2}=x_{e_2}+x_{e_3}=\ldots=x_{e_{n-1}}+x_{e_n}=0$,
and the fact that $n$ is odd implies that all of these entries are also~0.  Therefore, $x=0$ is the only solution
of $I_G\,x=0$, and so $\nullity(I_G)=0$.

Now, conversely, let it not be the case that each component of~$G$ is acyclic or contains exactly one cycle with that cycle having odd length.
Then $G$ contains an even-length cycle or two odd-length cycles connected by a path.
(It is assumed here that the two odd-length cycles either share no vertices, or else share
only one vertex, in which case the connecting path has length zero.  For if $G$ contains two odd-length cycles which share more than one
vertex, then~$G$ also has an even-length cycle, comprised of certain segments of the odd-length cycles.)
If~$G$ has an even-length cycle, then there exists $x\in\R^E$
which satisfies $I_G\,x=0$, where $x_e$ is alternately~$1$ and~$-1$ for each edge~$e$ along the cycle, and $x_e=0$ for each edge~$e$
not in the cycle.  If~$G$ contains two odd-length cycles connected by a path, then it can be seen that
there exists $x\in\R^E$ with $|x_e|=2$ for $e$ in the path or for~$e$ a loop, $|x_e|=1$ for~$e$ in a nonloop cycle,
and $x_e=0$ for~$e$ not in the path or either cycle, and where signs are assigned to the
nonzero entries of~$x$ so that $I_G\,x=0$.
(In the case in which~$G$ contains two loops connected by a path, there also exists
$x\in\R^E$ with $|x_e|=1$ if~$e$ is in the path or is one of the loops, and $x_e=0$ otherwise.)
Therefore, in each of these cases, there exists a nonzero
$x\in\R^E$ which satisfies $I_G\,x=0$, and so $\nullity(I_G)>0$.\hspace{\fill}$\Box$\\[2.5mm]
\emph{Alternative proof of Proposition~\ref{incnulllem}.} In this proof,
the arguments used in the alternative proof of Proposition~\ref{zeronullity} will be used to construct an explicit basis for the kernel of~$I_G$.
Let~$H$ be any spanning subgraph of $G$ with the property that, for each component~$C$ of $G$,
the subgraph of $H$ induced by the vertices of~$C$ is a
tree if~$C$ is bipartite, and is connected and contains exactly one cycle with that cycle having odd length if~$C$ is nonbipartite.
The existence of such a~$H$ is guaranteed by
the facts that a connected graph has a spanning tree and that a nonbipartite graph has an odd-length cycle.
It follows from the formulae relating numbers of edges and vertices in trees and in
connected graphs with exactly one cycle that
$|E'|=|V|-B$, and so $|E\setminus E'|=|E|-|V|+B$, where~$E'$ is the edge set of~$H$ and $B$ is the number of bipartite components of~$G$.
It can also be seen that, for each $f\in E\setminus E'$, the spanning subgraph of~$G$ with edge set $E'\cup\{f\}$ has
an even-length cycle containing~$f$, or two odd-length cycles (which share at most one vertex)
connected by a path, with one of those cycles containing~$f$.
Therefore, using the same argument as in the second part of the alternative proof of Proposition~\ref{zeronullity},
for each $f\in E\setminus E'$, there exists $x(f)\in\R^E$ satisfying the properties that $I_G\,x(f)=0$, $x(f)_f\ne0$, and
the edges $e$ for which~$x(f)_e\ne0$ are all contained in $E'\cup\{f\}$ and form either a single even-length cycle
or two odd-length cycles connected by a path. Choosing a particular such $x(f)$ for each $f\in E\setminus E'$,
it follows immediately that these are $|E|-|V|+B$ linearly independent elements of the kernel of~$I_G$.

It will now be shown that these vectors also span the kernel of~$I_G$.
First, let $y$ be any vector in the kernel of $I_G$,
and set $y'=\sum_{f\in E\setminus E'}y_f\,x(f)/x(f)_f$ (with $y'=0$ if $E\setminus E'=\emptyset$).
Then $y'_e=y_e$ for each $e\in E\setminus E'$ (since $x(f)_e=0$ for all $e\in E\setminus(E'\cup\{f\})$).
Also, $I_G\,y=I_G\,y'=I_G(y-y')=0$, and using the same argument as in the
first part of the alternative proof of Proposition~\ref{zeronullity}, it then follows that $(y-y')_e=0$ for each $e\in E'$, so that
$y=y'$.

Therefore, the vectors $x(f)$ with $f\in E\setminus E'$
form a basis of the kernel of $I_G$, and $\nullity(I_G)=|E|-|V|+B$.\hspace{\fill}$\Box$

\begin{proposition}\label{bipsollem}Consider an $a\in\R^V$, and for each bipartite component~$C$ of~$G$
let $(U_C,W_C)$ be a bipartition for~$C$.
Then a necessary and sufficient condition for there to exist an $x\in\R^E$ with $I_G\,x=a$ is that
\begin{equation}\label{bipsollemeq1}
\sum_{v\in U_C}a_v=\sum_{v\in W_C}a_v,\ \ \mbox{for each bipartite component }C\mbox{ of }G.\end{equation}\end{proposition}
Note that this result provides a necessary and sufficient condition for
there to exist an assignment of real numbers to the edges of~$G$
such that the sum of the numbers over all edges incident to any
vertex~$v$ is a prescribed real number~$a_v$.
\begin{proof}Consider any $x\in\R^E$, and any bipartite component~$C$ of~$G$.  It can be seen that
$\sum_{v\in U_C}\sum_{e\in\delta_G(v)}x_e=\sum_{v\in W_C}\sum_{e\in\delta_G(v)}x_e$, since each side is the sum of~$x_e$ over all edges~$e$ of~$C$.
Therefore,
\begin{equation}\label{bipsollemeq2}\textstyle\sum_{v\in W_C}\Bigl(\sum_{e\in\delta_G(v)}x_e-a_v\Bigr)-
\sum_{v\in U_C}\Bigl(\sum_{e\in\delta_G(v)}x_e-a_v\Bigr)=
\sum_{v\in U_C}a_v-\sum_{v\in W_C}a_v.\end{equation}
If $I_G\,x=a$ then the LHS of~\eqref{bipsollemeq2} immediately vanishes, and so~\eqref{bipsollemeq2} implies that~\eqref{bipsollemeq1} is satisfied.
Conversely, if~\eqref{bipsollemeq1} is satisfied then the RHS of~\eqref{bipsollemeq2} immediately vanishes, and so~\eqref{bipsollemeq2} enables
an equation $\sum_{e\in\delta_G(v)}x_e=a_v$ for a single vertex $v$ of each bipartite component of $G$ to be eliminated from the $|V|$ constituent
equations of $I_G\,x=a$.  This leaves $|V|-B$ equations,
where~$B$ is the number of bipartite components of $G$.  Using~\eqref{incnulllem2}, these remaining equations are linearly independent, and
therefore have a solution.\end{proof}

\begin{proposition}\label{uniquesollem}
Consider an $a\in\R^V$, and for each bipartite component~$C$ of~$G$
let $(U_C,W_C)$ be a bipartition for~$C$.
\begin{list}{(\roman{listnumber})}{\usecounter{listnumber}\setlength{\labelwidth}{6mm}\setlength{\leftmargin}{9mm}\setlength{\itemsep}{1mm}}
\item A necessary and sufficient condition for there to exist
a unique $x\in\R^E$ with $I_G\,x=a$ is that~\eqref{bipsollemeq1} is satisfied, and
each component of~$G$ either is acyclic or else
contains exactly one cycle with that cycle having odd length.
\item If the condition of (i) is satisfied, then the unique $x\in\R^E$ with $I_G\,x=a$
is given explicitly by
\begin{equation}\label{xe}x_e=k_e\sum_{v\in V_{G\setminus e}(t_e)}(-1)^{d_{G\setminus e}(v,t_e)}\,a_v,
\text{ for each }e\in E,
\end{equation}
where
\begin{align}k_e&=\begin{cases}\frac{1}{2},&e\text{ is an edge of a nonloop cycle of }G,\\
1,&\text{otherwise,}\end{cases}\\
\label{te}t_e&=\begin{cases}\mbox{the endpoint of $e$ furthest from }L,&e\text{ is an edge of a component}\\[-1.5mm]
&\text{of }G\text{ that contains a single}\\[-1.5mm]
&\text{cycle }L,\text{ but }e\text{ is not in }L,\\
\text{an arbitrarily-chosen endpoint of }e,&\text{otherwise,}\end{cases}\end{align}
$G\setminus e$ is the graph obtained by deleting edge~$e$ from $G$,
$V_{G\setminus e}(t_e)$ is the vertex set of the component of $G\setminus e$ which contains $t_e$,
and $d_{G\setminus e}(v,t_e)$ is the length of the (necessarily unique) path between~$v$ and~$t_e$ in $G\setminus e$.
\end{list}
\end{proposition}
Note that the path between~$v$ and~$t_e$ in~$G\setminus e$ is unique since,
for all cases of~\eqref{xe},~$v$ and~$t_e$ are vertices of a component of~$G\setminus e$ which
is acyclic.

It can also be checked that if there is choice for $t_e$ in~\eqref{te}, which occurs if $e$ is an an edge of a tree or of a nonloop cycle,
then the RHS of~\eqref{xe} is independent of that choice.
For example, consider an edge $e$ of a tree with vertex set~$T$, let the endpoints of~$e$ be~$u$ and~$w$,
and denote the length of the path between any two vertices~$v$ and~$v'$ in~$T$ as $d_T(v,v')$. Then
the RHS of~\eqref{xe} is $\sum_{v\in V_{G\setminus e}(u)}(-1)^{d_T(v,u)}\,a_v$ for the choice $t_e=u$,
and $\sum_{v\in V_{G\setminus e}(w)}(-1)^{d_T(v,w)}\,a_v=
\sum_{v\in V_{G\setminus e}(w)}(-1)^{d_T(v,u)-1}\,a_v=
-\sum_{v\in V_{G\setminus e}(w)}(-1)^{d_T(v,u)}\,a_v$ for the choice $t_e=w$.
Therefore, since $T=V_{G\setminus e}(u)\cupdot V_{G\setminus e}(w)$,
the difference between the previous expressions is $\sum_{v\in T}(-1)^{d_T(v,u)}\,a_v$, which vanishes due to the condition~\eqref{bipsollemeq1}
satisfied by $a$.
\begin{proof}The validity of (i) follows from Propositions~\ref{zeronullity} and~\ref{bipsollem}.

Now let the condition of (i) be satisfied.  Then it can be verified directly that $x$, as given by~\eqref{xe}, satisfies
$\sum_{e\in\delta_G(v)}x_e=a_v$ for each $v\in V$, and hence that (ii) is valid.
The nature of the verification process depends on whether~$v$ is
a vertex of a tree,~$v$ is a vertex with a loop attached,~$v$ is a vertex of a nonloop cycle,
or~$v$ is a vertex of a component which contains a cycle but with~$v$ not in the cycle.
The details for the first of these cases will now be given explicitly, with those for the others being similar.
So, let~$v$ be a vertex of a tree with vertex set $T$, for each $e\in\delta_G(v)$ choose $t_e$ to be the endpoint of~$e$ other than~$v$,
and (as before) denote the length of the path between vertices~$v'$ and~$v''$ in~$T$ as $d_T(v',v'')$.
Then
\begin{align*}\displaystyle
\sum_{e\in\delta_G(v)}x_e&=\sum_{e\in\delta_G(v)}\sum_{u\in V_{G\setminus e}(t_e)}(-1)^{d_T(u,t_e)}\,a_u
=\sum_{e\in\delta_G(v)}\sum_{u\in V_{G\setminus e}(t_e)}(-1)^{d_T(u,v)-1}\,a_u\\
&\displaystyle=\:-\!\sum_{e\in\delta_G(v)}\sum_{u\in V_{G\setminus e}(t_e)}(-1)^{d_T(u,v)}\,a_u
\,=\:a_v-\sum_{u\in T}(-1)^{d_T(u,v)}\,a_u\,=\:a_v,\end{align*}
where the second-last equality follows from the fact that $T\setminus\{v\}$ is the union of
the mutually disjoint sets $V_{G\setminus e}(t_e)$ over all~$e\in\delta_G(v)$,
and the last equality follows from the condition~\eqref{bipsollemeq1}
satisfied by $a$.\end{proof}

Note that if $a$ and $G$, as given in Proposition~\ref{bipsollem}, satisfy the condition~\eqref{bipsollemeq1}
in that proposition (but not necessarily the condition of (i) in Proposition~\ref{uniquesollem}),
then a (not necessarily unique) $x\in\R^E$ with $I_G\,x=a$
(whose existence is guaranteed by Proposition~\ref{bipsollem}) can be obtained as follows.  First, let $H$ be a spanning subgraph
of $G$, chosen to satisfy the same properties as the $H$ used in the alternative proof of Proposition~\ref{incnulllem}.  Then it follows
from~(ii) of Proposition~\ref{uniquesollem} (using $H$ instead of $G$) that there exists a unique $x'\in\R^{E'}$ with $I_{H}\,x'=a$, where $E'$ is the edge
set of $H$.  The required $x\in\R^E$ is then given by $x_e=x'_e$ for each $e\in E'$, and $x_e=0$ for each $e\in E\setminus E'$.

\begin{proposition}\label{onenullity}
The nullity of the incidence matrix of $G$ is~$1$ if and only if $G$ has a component~$C$ such that
each component of~$G$ other than $C$ either is acyclic or else contains exactly one cycle with that cycle having odd length,
while the cycle content of $C$ is one of the following:
\begin{itemize}
\item exactly one cycle, with that cycle having even length, or
\item exactly two cycles, with at least one of those cycles having odd length, or
\item exactly one even-length and exactly two odd-length cycles, with any two of those cycles sharing at least one edge.
\end{itemize}
\end{proposition}
Note that for the case of bipartite $G$, it follows from this proposition,
and the fact that a bipartite graph does not contain any odd-length cycles, that
the nullity of the incidence matrix of $G$ is~$1$ if and
only if~$G$ contains exactly one cycle.
\begin{proof}The structure of this proof is similar to that of the proof of Proposition~\ref{zeronullity}.
The fact that the kernel of $I_G$ is the direct sum of the kernels of the incidence matrices of its components now implies
that $\nullity(I_G)=1$ if and only if there exists a component~$C$ of~$G$ such that
$\nullity(I_C)=1$, and $\nullity(I_{C'})=0$ for each component~$C'$ of~$G$ other than~$C$.
Using Proposition~\ref{zeronullity}, the equation $\nullity(I_{C'})=0$, for each component~$C'$ other than~$C$, is equivalent to the condition
that~$C'$ either is acyclic or else contains exactly one cycle with that cycle having odd length.
Using Proposition~\ref{incnulllem}, the equation $\nullity(I_C)=1$ is equivalent to
$|E_C|=|V_C|$ if $C$ is bipartite, or $|E_C|=|V_C|+1$ if $C$ is nonbipartite, where~$E_C$ and~$V_C$ are the edge and vertex sets of~$C$.
Using the facts that a connected graph~$C$ satisfies $|E_C|=|V_C|$ if and only if $C$ contains exactly one cycle,
and satisfies $|E_C|=|V_C|+1$ if and only if $C$ contains either exactly two cycles
or else exactly three cycles, with any two of the three cycles sharing at least one edge
(where the latter fact can be derived straightforwardly),
it now follows that $\nullity(I_C)=1$ if and only if $C$ either is bipartite and contains exactly one cycle,
or else is nonbipartite and contains exactly two cycles
or exactly three cycles, with any two of the three cycles sharing at least one edge.
Finally, the conditions on the parities of the lengths of the cycles of~$C$, as given in the statement of the proposition, follow from
the fact that a graph is bipartite if and only if it does not contain any odd-length cycles.
\end{proof}

\section{Relevant results for polytopes}\label{poly}
In this section, definitions are given for supports, and some relevant standard
results involving faces, dimensions, vertices and edges of polytopes are obtained.
These general results will be applied to~$\P(G,b)$ in Section~\ref{facesPGb}.

Let~$N$ be a finite set, and define the support of any $X\subset\R^N$ as
\begin{equation}\label{supp}\supp(X):=\{i\in N\mid\mbox{there exists }x\in X\mbox{ with }x_i\ne0\},\end{equation}
and the support of any $x\in\R^N$ as
\begin{align}\notag\supp(x)&:=\supp(\{x\})\\
\label{suppx}&\,=\{i\in N\mid x_i\ne0\}.\end{align}

Some simple but useful properties of supports, which follow immediately from~\eqref{supp}, are that, for any $X_1,X_2\subset\R^N$,
\begin{equation}\label{suppimpl}X_1\subset X_2\text{ implies }\supp(X_1)\subset\supp(X_2),\end{equation}
and that, for any set $\mathcal{R}$ of subsets of $\R^N$,
\begin{align}\label{suppint}\supp\biggl(\bigcap_{X\in\mathcal{R}}X\biggr)&\subset\bigcap_{X\in\mathcal{R}}\supp(X),\\
\label{suppunion}\supp\biggl(\bigcup_{X\in\mathcal{R}}X\biggr)&=\bigcup_{X\in\mathcal{R}}\supp(X),\end{align}
with empty intersections in~\eqref{suppint} taken as $\bigcap_{X\in\emptyset}X=\R^N$ and $\bigcap_{X\in\emptyset}\supp(X)=N$.
It follows that, for any $X\subset\R^N$,
\begin{equation}\label{suppXsuppx}\supp(X)=\bigcup_{x\in X}\supp(x).\end{equation}

Now let $M$ be a further finite set, $A$ be a real matrix with rows and columns indexed by~$M$ and~$N$ respectively,~$a$ be a
vector in~$\R^M$, and $P$ be a polytope which can be written as
\begin{equation}\label{polytope}P=\{x\in\R^N\mid x_i\ge0\mbox{ for each }i\in N, \ A\,x=a\}.\end{equation}

The following results, Propositions~\ref{facelem},~\ref{vertlem}
and~\ref{edgelem}, will provide information regarding
the faces, vertices and edges, respectively, of the polytope $P$ given by~\eqref{polytope}.
These results are all closely related to standard results for polyhedra.

\begin{proposition}\label{facelem}
Let $F$ be a nonempty face of the polytope $P$ given by~\eqref{polytope}, and define $N'=\supp(F)$,~$A'$ to be the
submatrix of $A$ obtained by restricting the columns of $A$ to those indexed by~$N'$, and
\begin{equation}\label{facelem1}F':=\{y\in\R^{N'}\mid y_i\ge0\mbox{ for each }i\in N', \ A'\,y=a\}.\end{equation}
Then:
\begin{list}{(\roman{listnumber})}{\usecounter{listnumber}\setlength{\labelwidth}{6mm}\setlength{\leftmargin}{9mm}\setlength{\itemsep}{1mm}}
\item $F'$ is a polytope which is affinely isomorphic to $F$.
\item The dimension of $F$ (and also the dimension of $F'$) equals the nullity of $A'$.
\end{list}
\end{proposition}
\begin{proof}
The fact that any nonempty face of a polyhedron can be obtained by setting a subset of the polyhedron's
defining inequalities to equalities (see, for example, Korte and Vygen~\cite[Prop.~3.4]{KorVyg12}, Schrijver~\cite[Sec.~8.2, Eq.~(11)]{Sch86}
or Schrijver~\cite[Eq.~(5.16)]{Sch03})
means that there exists some $N''\subset N$ for which
\begin{equation*}F=\{x\in\R^N\mid x_i\ge0\mbox{ for each }i\in N'',\ x_i=0\text{ for each }i\in N\setminus N'',\ A\,x=a\}.\end{equation*}
It can be seen, using $N'=\supp(F)$, that $N'\subset N''$, and hence that
\begin{equation*}F=\{x\in\R^N\mid x_i\ge0\mbox{ for each }i\in N',\ x_i=0\text{ for each }i\in N\setminus N',\ A\,x=a\}.\end{equation*}

Now consider the affine mapping from any $x\in\R^N$ to~$y\in\R^{N'}$ in which $y_i=x_i$ for each $i\in N'$,
and the affine mapping from any $y\in\R^{N'}$ to~$x\in\R^N$ in which
$x_i=y_i$ for each $i\in N'$, and $x_i=0$ for each $i\in N\setminus N'$.
Then these (essentially trivial) mappings, with their domains restricted to $F$ and $F'$ respectively, are mutual inverses,
and so~$F'$ is a polytope which is affinely isomorphic to $F$, thus confirming (i).

It follows immediately that $\dim(F)=\dim(F')$.
It can also be seen that~$N'=\supp(F')$, so that
none of the inequalities $y_i\ge0$ in~\eqref{facelem1} is an implicit equality.
Hence, using the fact that the dimension of a nonempty polyhedron is the nullity of the matrix associated with those of the polyhedron's
defining inequalities which are implicit equalities
(see, for example, Schrijver~\cite[Sec.~8.2, Eq.~(9)]{Sch86} or Schrijver~\cite[Thm.~5.6]{Sch03}),
it follows that $\dim(F')=\nullity(A')$, thus confirming (ii).
\end{proof}

\begin{proposition}\label{vertlem}
Let $u$ be an element of the polytope $P$ given by~\eqref{polytope}, and define~$A'$ to be the
submatrix of $A$ obtained by restricting the columns of $A$ to those indexed by~$\supp(u)$.
Then~$u$ is a vertex of $P$ if and only if $\nullity(A')=0$.\end{proposition}
\begin{proof}If $u$ is a vertex of $P$, then $\{u\}$ is a face of $P$,
and so, using (ii) of Proposition~\ref{facelem}, $0=\dim(\{u\})=\nullity(A')$.
Conversely, if $\nullity(A')=0$, then the equation $A'\,y=a$ has the unique solution
$y\in\R^{N'}$ given by $y_i=u_i$ for each $i\in N'$, where $N'=\supp(u)$.
It then follows that $\{x\in P\mid x_i=0\mbox{ for each }i\in N\setminus N'\}=\{u\}$,
which implies that $\{u\}$ is face of~$P$, and hence that $u$ is a vertex of~$P$.
\end{proof}

\begin{corollary}\label{vertcor}
Let $u$ be a vertex of the polytope $P$ given by~\eqref{polytope}.  Then $|\supp(u)|\le\rank(A)$.\end{corollary}
\begin{proof}Using Proposition~\ref{vertlem}, $\nullity(A')=0$, where~$A'$ is the submatrix of $A$ obtained by restricting the columns
of~$A$ to those indexed by~$\supp(u)$. Therefore, $|\supp(u)|=\rank(A')\le\rank(A)$.
\end{proof}
Note that, in some contexts in the literature,
a vertex $u$ of the polytope $P$ given by~\eqref{polytope} is referred to as nondegenerate or degenerate according
to whether or not $|\supp(u)|=\rank(A)$.

\begin{proposition}\label{edgelem}
Let $u$ and $w$ be distinct vertices of the polytope $P$ given by~\eqref{polytope}, and define~$A'$ to be the
submatrix of $A$ obtained by restricting the columns of $A$ to those indexed by~$\supp(\{u,w\})$.
Then~$u$ and~$w$ are the vertices of an edge of $P$ if and only if $\nullity(A')=1$.\end{proposition}
Note that, using \eqref{suppunion}, $\supp(\{u,w\})=\supp(u)\cup\supp(w)$.
\begin{proof}It can be seen that
$\supp(\{u,w\})=\supp([u,w])$, where $[u,w]$ is the closed line segment between $u$ and~$w$.  Therefore, if~$u$ and~$w$ are the
vertices of an edge of $P$, then $[u,w]$ is a face of $P$, and so,
using~(ii) of Proposition~\ref{facelem}, $1=\dim([u,w])=\nullity(A')$.
Conversely, if $\nullity(A')=1$, then the equation $A'\,y=a$ for $y\in\R^{N'}$ has the general solution
$y=\lambda u'+(1-\lambda)w'$, where $N'=\supp(\{u,w\})$,
$u',w'\in\R^{N'}$ are given by $u_i'=u_i$ and $w'_i=w_i$ for each $i\in N'$, and $\lambda\in\R$ is arbitrary.
It then follows, since $u$ and $w$ are vertices of $P$, that $\{x\in P\mid x_i=0\mbox{ for each }i\in N\setminus N'\}=[u,w]$,
which implies that $[u,w]$ is a face of~$P$, and hence that $u$ and $w$ are the vertices of an edge of~$P$.
\end{proof}

\section{Relevant results for polytope face lattices}\label{facelatt}
In this section, some relevant, essentially standard, results concerning polytope face lattices are outlined.
In particular, for a polytope $P$ given by~\eqref{polytope} (with $a\ne0$), three isomorphic lattices (all partially ordered by set
inclusion) are considered, namely the
face lattice of $P$, the lattice of vertex sets of the faces of $P$, and the lattice of supports of the faces of $P$
(where these are denoted as $\F(P)$, $\V(P)$ and $\S(P)$, respectively).
Various expressions will be presented for elements of these
lattices, for the meet and join of subsets of these lattices, and for isomorphisms among the lattices.
The general results of this section will be applied to~$\P(G,b)$ in Section~\ref{facelattPGb}.

In contrast to the results in other sections of this paper,
some details of the proofs of the results of this section will be omitted.  However, full proofs can
be obtained straightforwardly using standard polyhedron and polytope theory,
as given, for example, in the books of Br\o nsted~\cite{Bro83}, Gr\"{u}nbaum~\cite{Gru03},
Korte and Vygen~\cite[Ch.~3]{KorVyg12},
Schrijver~\cite[Ch.~8--9]{Sch86}, Yemelichev, Kovalev and Kravtsov~\cite{YemKovKra84}, or Ziegler~\cite{Zie95}.  The
simple properties~\eqref{suppimpl}--\eqref{suppunion} of supports are also useful for some of these proofs.

Let $N$ be a finite set. For any subset $X$ of $\R^N$, and any polytope $P$ in $\R^N$, denote (as already indicated in
Section~\ref{convnot}) the convex hull of~$X$ as $\conv(X)$, the
set of vertices of~$P$ as $\vert(P)$,
the set of facets of~$P$ as $\facets(P)$, and the face lattice of~$P$ as~$\F(P)$.

Some of the statements which follow will involve the notation $\subset'$, where this means that such a statement
is valid if $\subset'$ is taken to be $\subset$, and also valid if $\subset'$ is taken to be $=$.

Now consider a given polytope $P$. Then the face lattice
$\F(P)$ is the set of all faces of~$P$ (including $\emptyset$ and~$P$) partially ordered by set inclusion,
where, for any $\H\subset\F(P)$, the infimum or meet of~$\H$ is the intersection of all the faces in $\H$,
and the supremum or join of $\H$ is the intersection of all those faces, or alternatively facets, of $P$ which contain
each face in $\H$, i.e.,
\begin{align}\inf(\H)&=\bigcap_{F\in\H}F,\\
\sup(\H)&=\bigcap_{\substack{F\in\F(P)\\\cup_{F'\in\H}F'\subset F}}\!\!F\;
=\bigcap_{\substack{F\in\facets(P)\\\cup_{F'\in\H}F'\subset F}}\!\!F,\end{align}
with an intersection over $\emptyset$ taken to be $P$.

Define $\V(P)$ to be the set of vertex sets of the faces of $P$, i.e.,
\begin{equation}\V(P):=
\{\vert(F)\mid F\in\F(P)\}.\end{equation}
This can also be written as
\begin{align}\notag\V(P)&\textstyle=\bigl\{\bigcap_{F\in\H}\vert(F)\bigm|\H\subset\facets(P)\bigr\}\\
\label{VP}&\textstyle=\Bigl\{U\subset\vert(P)\Bigm|\bigcap_{\substack{F\in\facets(P)\\U\subset\vert(F)}}\vert(F)\subset'U\Bigr\},\end{align}
with intersections over $\emptyset$ taken to be $\vert(P)$.

The mapping from each $F\in\F(P)$ to $\vert(F)=\vert(P)\cap F\in\V(P)$
is bijective, with an inverse which maps each $U\in\V(P)$ to $\conv(U)=\sup(\{\{u\}\mid u\in U\})=
\bigcap_{\substack{F'\in\F(P)\\U\subset F'}}F'=\bigcap_{\substack{F'\in\facets(P)\\U\subset F'}}F'\in\F(P)$,
and with the property that, for any $F_1,F_2\in\F(P)$, $F_1\subset F_2$ if and only if $\vert(F_1)\subset\vert(F_2)$.
Thus, $\V(P)$ is a lattice, partially ordered by set inclusion, which is isomorphic to the face lattice~$\F(P)$.
Also, for any $\mathcal{W}\subset\V(P)$, the meet and join in this lattice are
\begin{align}\inf(\mathcal{W})&=\bigcap_{U\in\mathcal{W}}U,\\
\sup(\mathcal{W})&=\bigcap_{\substack{U\in\V(P)\\\cup_{U'\in\mathcal{W}}U'\subset U}}\!\!U\;
=\bigcap_{\substack{F\in\facets(P)\\\cup_{U'\in\mathcal{W}}U'\subset\vert(F)}}\!\vert(F),\end{align}
with an intersection over $\emptyset$ taken to be $\vert(P)$.

Now let $P$ be a polytope which can be written as~\eqref{polytope},
for some real matrix $A$ with rows and columns indexed by finite sets~$M$ and~$N$ respectively,
and some real nonzero vector~$a$ with entries indexed by~$M$.
Note that the condition $a\ne0$, or equivalently $P\ne\{0\}$, is needed to ensure the validity of some of the results which follow.
For example, this condition is needed to guarantee that~$\emptyset$ is contained in the RHS of~\eqref{faces}.

The face lattice of $P$ can now be written as
\begin{equation}\label{faces}\F(P)=\bigl\{\{x\in P\mid x_i=0\text{ for each }i\in N'\}\bigm| N'\subset N\bigr\}.\end{equation}

Some useful properties involving supports, as defined in~\eqref{supp}--\eqref{suppx}, and faces of $P$ are that, for any $X\subset P$,
\begin{equation}\label{suppX}\supp(X)=\supp\Biggl(\bigcap_{\substack{F\in\F(P)\\X\subset F}}F\Biggr)
=\supp\Biggl(\bigcap_{\substack{F\in\facets(P)\\X\subset F}}F\Biggr),\end{equation}
and so, for any $x\in P$,
\begin{equation}\label{suppXx}\supp(x)=\supp\Biggl(\bigcap_{\substack{F\in\F(P)\\x\in F}}F\Biggr)
=\supp\Biggl(\bigcap_{\substack{F\in\facets(P)\\x\in F}}F\Biggr),\end{equation}
with an intersection over $\emptyset$ taken to be $P$.

Also, for any $X\subset P$ and $F\in\F(P)$,
\begin{equation}\label{suppiff}X\subset F\text{ if and only if }\supp(X)\subset\supp(F),\end{equation}
and so, for any $x\in P$ and $F\in\F(P)$,
\begin{equation}\label{suppiff2}x\in F\text{ if and only if }\supp(x)\subset\supp(F).\end{equation}
Note that the `only if' part of~\eqref{suppiff} follows immediately from~\eqref{suppimpl}.

A condition for the vertices of $P$ in terms of supports is that
an element~$u$ of~$P$ is a vertex of~$P$ if and only if there is no other element of $P$
whose support is contained in, or alternatively equal to, the support of~$u$, i.e., for any $u\in P$,
\begin{align}\notag u\in\vert(P)&\text{ if and only if }
\bigl\{x\in P\bigm|\supp(x)=\supp(u)\bigr\}\subset'\{u\}\\
\label{vertcond}&\text{ if and only if }\bigl\{x\in P\bigm|\supp(x)\subset\supp(u)\bigr\}\subset'\{u\}.\end{align}

Furthermore, for any $U\subset\vert(P)$,
\begin{equation}\{u\in\vert(P)\mid\supp(u)\subset\supp(U)\}=
\bigcap_{\substack{F\in\facets(P)\\U\subset\vert(F)}}\vert(F),\end{equation}
from which it follows, using~\eqref{VP}, that $\V(P)$ can be expressed in terms of supports as
\begin{equation}\label{VP2}\V(P)=\bigl\{U\subset\vert(P)\bigm|\{u\in\vert(P)\mid\supp(u)\subset\supp(U)\}\subset'U\bigr\}.\end{equation}

By considering two-element subsets $U$ of $\vert(P)$ in~\eqref{VP2}, it follows that
a condition for the edges of $P$ is that, for any distinct vertices $u$ and $w$ of $P$,
\begin{multline}\label{edgecond}\qquad\quad u\text{ and }w\text{ are the vertices of an edge of $P$ if and only if}\\
\bigl\{y\in\vert(P)\bigm|\supp(y)\subset\supp(\{u,w\})\bigr\}\subset'\{u,w\}.\qquad\end{multline}

Now define $\S(P)$ as the set of supports of the faces of $P$, i.e.,
\begin{equation}\label{SPdef}\S(P):=\{\supp(F)\mid F\in\F(P)\}.\end{equation}

This can also be written as
\begin{align}\notag\S(P)&=\{\supp(X)\mid X\subset P\}\\
\notag&=\{\supp(x)\mid x\in P\}\cup\{\emptyset\}\\
\notag&=\{\supp(U)\mid U\subset\vert(P)\}\\
\label{SP}&\textstyle=\Bigl\{S\subset N\Bigm|S\subset'\bigcup_{\substack{u\in\vert(P)\\\supp(u)\subset S}}\supp(u)\Bigr\}.\end{align}

It follows from~\eqref{suppiff2} that the mapping from each $F\in\F(P)$ to $\supp(F)\in\S(P)$
is bijective, with an inverse which maps each $S\in\S(P)$ to $\{x\in P\mid\supp(x)\subset S\}=
\{x\in P\mid x_i=0\text{ for all }i\in N\setminus S\}$.
Furthermore, it follows from~\eqref{suppiff} that, for any~$F_1,F_2\in\F(P)$, $F_1\subset F_2$ if and only if $\supp(F_1)\subset\supp(F_2)$.
Thus, $\S(P)$ is a lattice, partially ordered by set inclusion, which is isomorphic to the face lattice~$\F(P)$.
Also, for any $\mathcal{T}\subset\S(P)$, the meet and join in this lattice are
\begin{align}\label{infT}\inf(\mathcal{T})&=\bigcup_{\substack{S\in\S(P)\\S\subset\cap_{S'\in\mathcal{T}}S'}}\!\!S\;
=\bigcup_{\substack{u\in\vert(P)\\\supp(u)\subset\cap_{S'\in\mathcal{T}}S'}}\!\supp(u),\\
\sup(\mathcal{T})&=\bigcup_{S\in\mathcal{T}}S.\end{align}
In~\eqref{infT}, $\bigcap_{S'\in\mathcal{\emptyset}}S'$ can be taken as $N$ or as $\supp(P)$, giving $\inf(\emptyset)=\supp(P)$.

Combining the isomorphisms between $\F(P)$ and $\V(P)$, and between $\F(P)$ and $\S(P)$, gives
an isomorphism between $\V(P)$ and $\S(P)$ in which
each $U\in\V(P)$ is mapped to $\supp(U)\in\S(P)$ and, inversely, in which
each $S\in\S(P)$ is mapped to $\{u\in\vert(P)\mid\supp(u)\subset S\}\in\V(P)$.

The isomorphisms among the lattices $\F(P)$, $\V(P)$ and $\S(P)$ are summarized in Figure~\ref{iso}.

\begin{figure}[h]\centering
\psset{unit=9mm}\pspicture(-0.5,-0.9)(12,5.3)
\psline[linewidth=0.5pt,arrows=->,arrowsize=5pt](0,0)(5,5)\psline[linewidth=0.5pt,arrows=<-,arrowsize=5pt](0.5,0)(5.2,4.7)
\psline[linewidth=0.5pt,arrows=->,arrowsize=5pt](6.3,4.7)(11,0)\psline[linewidth=0.5pt,arrows=<-,arrowsize=5pt](6.5,5)(11.5,0)
\psline[linewidth=0.5pt,arrows=->,arrowsize=5pt](1,-0.25)(10.5,-0.25)\psline[linewidth=0.5pt,arrows=<-,arrowsize=5pt](1,-0.6)(10.5,-0.6)
\rput(11.3,-0.45){$\S(P)$}\rput(0.3,-0.45){$\V(P)$}\rput(5.75,5){$\F(P)$}
\rput[l](8.85,3){$\scriptstyle S\mapsto\{x\in P\mid\supp(x)\subset S\}$} \rput[r](8.8,2){$\scriptstyle F\mapsto\supp(F)$}
\rput[l](2.8,2){$\scriptstyle F\mapsto\vert(F)$} \rput[r](2.7,3){$\scriptstyle U\mapsto\conv(U)$} \rput[b](5.75,-0.1){$\scriptstyle U\mapsto\supp(U)$}
\rput[t](5.75,-0.75){$\scriptstyle S\mapsto\{u\in\vert(P)\mid\supp(u)\subset S\}$}\endpspicture
\caption{Isomorphisms among the lattices $\F(P)$, $\V(P)$ and $\S(P)$}\label{iso}\end{figure}

Finally, note that the results of this section can easily be modified so as to be valid for the more general case in which
the polytope $P$ has the form $P=\bigcap_{i\in M}K_i$ instead of~\eqref{polytope},
and the supports of any $X\subset\R^N$ and any $x\in\R^N$ are defined to be
$\supp(X)=\{i\in M\mid X\not\subset H_i\}$ and $\supp(x)=\supp(\{x\})=\{i\in M\mid x\notin H_i\}$ instead of~\eqref{supp}--\eqref{suppx},
where~$M$ is a finite set and, for each $i\in M$, $K_i$ is a given closed halfspace in $\R^N$ whose bounding hyperplane is~$H_i$, with
it being assumed that $\bigcap_{i\in M}H_i=\emptyset$.  In particular, the modifications are that~\eqref{faces} becomes
$\F(P)=\bigl\{\bigl(\bigcap_{i\in M'}K_i\bigr)\,\bigcap\,\bigl(\bigcap_{i\in M\setminus M'}H_i\bigr)\bigm|M'\subset M\bigr\}$
(with an intersection over~$\emptyset$ taken to be~$\R^N$), the isomorphism from
$S\in\S(P)$ to $F\in\F(P)$ becomes $F=\{x\in P\mid\supp(x)\subset S\}=
\bigl(\bigcap_{i\in S}K_i\bigr)\,\bigcap\,\bigl(\bigcap_{i\in M\setminus S}H_i\bigr)$,
and~$N$ in~\eqref{SP} is replaced by~$M$, with all other statements in this section remaining unchanged.

\section{Results for the faces, dimension, vertices and edges of $\mathcal{P}(G,b)$}\label{facesPGb}
In this section, graphs associated with subsets and elements of $\P(G,b)$ are defined,
and some of the main results of this paper, concerning the faces, dimension, vertices and edges of $\P(G,b)$, are obtained
by combining general results for graphs from Section~\ref{inc} with general results for polytopes from Section~\ref{poly}.

Using the definitions~\eqref{supp}--\eqref{suppx} of supports (with $N$ taken to be $E$), define the graph of any $X\subset\P(G,b)$ as
\begin{align}\notag\gr(X)&:=\text{the spanning subgraph of }G\text{ with edge set }\supp(X)\\
\notag&\;=\text{the spanning subgraph of }G\text{ with edge set}\\[-1mm]
\label{grX}&\qquad\qquad\qquad\qquad\qquad\{e\in E\mid\text{there exists }x\in X\mbox{ with }x_e>0\},\end{align}
and define the graph of any $x\in\P(G,b)$ as
\begin{align}\notag\gr(x)&:=\gr(\{x\})\\
\notag&\;=\text{the spanning subgraph of }G\text{ with edge set }\supp(x)\\
\label{grx}&\;=\text{the spanning subgraph of }G\text{ with edge set }\{e\in E\mid x_e>0\}.\end{align}

It will also be useful, for Section~\ref{facelattPGb}, to denote the set
of graphs of subsets of $\P(G,b)$ as~$\G(G,b)$, i.e.,
\begin{equation}\label{GGbdef}\G(G,b):=\{\gr(X)\mid X\subset\P(G,b)\},\end{equation}
and to refer to the graphs in this set as the graphs of $\P(G,b)$.
(Note that any element of~$\G(G,b)$ is \emph{a} graph of~$\P(G,b)$, whereas the particular element $\gr(\P(G,b))$ is \emph{the}
graph of~$\P(G,b)$.)

The standard definitions of graph union, intersection and containment will be applied to
spanning subgraphs of $G$ in Section~\ref{facelattPGb}, and, to some extent, in this section.
In particular, for any set $\H$ of spanning subgraphs of $G$, the union and intersection of the graphs in $\H$ are given by
\begin{align}\label{graphint}\textstyle\bigcap_{H\in\H}H&=\textstyle\text{the spanning subgraph of }G\text{ with edge set }\bigcap_{H\in\H}E_H,\\
\label{graphunion}\textstyle\bigcup_{H\in\H}H&=\textstyle\text{the spanning subgraph of }G\text{ with edge set }\bigcup_{H\in\H}E_H,\end{align}
where $E_H$ denotes the edge set of $H$, and $\bigcap_{H\in\emptyset}E_H$ is taken to be $E$ (so that $\bigcap_{H\in\emptyset}H=G$).
Similarly, for spanning subgraphs $H_1$ and $H_2$ of $G$ with edge sets $E_{H_1}$ and $E_{H_2}$, graph containment is given by
\begin{equation}\label{subgraph}H_1\subset H_2\text{ if and only if }E_{H_1}\subset E_{H_2}.\end{equation}

For example, it follows from~\eqref{suppXsuppx} and~\eqref{graphunion} that, for any $X\subset\P(G,b)$,
\begin{equation}\label{grXgrx}\textstyle\gr(X)=\bigcup_{x\in X}\gr(x).\end{equation}

The main results of this section will now be obtained.  These include results
which provide formulae for the dimensions of $\P(G,b)$ and its faces (Theorem~\ref{facedim},
and Corollaries~\ref{dim},~\ref{dimcor1} and~\ref{dimcor2}), characterize the elements of $\P(G,b)$ which are vertices of $\P(G,b)$
(Theorem~\ref{vertth} and Corollary~\ref{bipvertcor}), give an explicit formula for a vertex of $\P(G,b)$ in terms of its
graph (Theorem~\ref{explvert}), and characterize the pairs of distinct vertices of $\P(G,b)$ which form edges of $\P(G,b)$
(Theorem~\ref{edgeth} and Corollary~\ref{bipedgeth}).  Certain cases of some of these results correspond to previously-known results
for the matrix classes $\N_{\le Z}(R)$, $\N(R)$, $\N_{\le Z}(R,S)$ or~$\N(R,S)$ (using the notation
discussed in Section~\ref{mat}).

\begin{theorem}\label{facedim}
Let $F$ be a nonempty face of $\P(G,b)$. Then:
\begin{list}{(\roman{listnumber})}{\usecounter{listnumber}\setlength{\labelwidth}{6mm}\setlength{\leftmargin}{9mm}\setlength{\itemsep}{1mm}}
\item $F$ is affinely isomorphic to $\P(\gr(F),b)$.
\item The dimension of $F$ is $|\supp(F)|-|V|+B$, where~$B$ is the number of bipartite components of $\gr(F)$.
\end{list}
\end{theorem}
\begin{proof}
Using~\eqref{IMF},~\eqref{grX}, and both parts of Proposition~\ref{facelem}
(taking $P$, $M$, $N$, $A$ and $a$ to be $\P(G,b)$,~$V$,~$E$,~$I_G$ and~$b$ respectively,
so that $A'=I_{\gr(F)}$), it follows that $F$ is affinely isomorphic to $\P(\gr(F),b)$,
and that $\dim(F)=\nullity(I_{\gr(F)})$.  The expression for $\dim(F)$ in (ii)
is then given by Proposition~\ref{incnulllem} (taking the graph in that proposition to be~$\gr(F)$).
\end{proof}

\begin{corollary}\label{dim}If $\P(G,b)$ is nonempty, then its dimension is $|\supp(\P(G,b))|-|V|+B$,
where~$B$ is the number of bipartite components of $\gr(\P(G,b))$.
\end{corollary}
A case of this result applied to $\N_{\le Z}(R,S)$ is given by, for example, Brualdi~\cite[Lem.\ 8.4.3(ii)]{Bru06},
and a case applied to $\N(R,S)$ is given by, for example,
Brualdi~\cite[Thm.~8.1.1]{Bru06}, Klee and Witzgall~\cite[Thm.~1]{KleWit68},
Schrijver~\cite[Thm.~21.16]{Sch03}, and
Yemelichev, Kovalev and Kravtsov~\cite[Ch.~6, Prop.~1.1]{YemKovKra84}.
\begin{proof}This result follows from (ii) of Theorem~\ref{facedim}, by taking $F$ to be~$\P(G,b)$.\end{proof}
Note that, assuming the validity of (i) of Theorem~\ref{facedim}, Corollary~\ref{dim} and (ii) of Theorem~\ref{facedim} are
equivalent, since (ii) of Theorem~\ref{facedim} could be obtained from Corollary~\ref{dim} by taking $G$ in
that corollary to be $\gr(F)$.

\begin{corollary}\label{dimcor1}
If $\P(G,b)$ contains a strictly positive element, then its dimension is $|E|-|V|+B$, where~$B$ is the number of
bipartite components of $G$.
\end{corollary}
\begin{proof}This result follows from Corollary~\ref{dim},
and the fact that if $\P(G,b)$ contains a completely positive element, then $\gr(\P(G,b))=G$.
\end{proof}

\begin{corollary}\label{dimcor2}
Let $G$ be bipartite and planar.  Then the dimension of a nonempty face $F$ of $\P(G,b)$ is the number of bounded faces
in a planar embedding of $\gr(F)$.
\end{corollary}
\begin{proof}This result follows from (ii) of Theorem~\ref{facedim}, and Euler's formula for planar graphs
(which remains valid for graphs with multiple edges).\end{proof}

\begin{corollary}\label{vertsizecor}
Let $u$ be a vertex of $\P(G,b)$.  Then $|\supp(u)|\le|V|-B$, where $B$ is the number of bipartite components of $G$.\end{corollary}
A case of this result applied to $\N(R)$ is given by, for example, Brualdi~\cite[Cor.~8.2.2]{Bru06},
Converse and Katz~\cite[Lem.]{ConKat75}, and Lewin~\cite[Cor.~2]{Lew77},
and a case applied to~$\N(R,S)$ is given by, for example,
Brualdi~\cite[Thm.~8.1.3]{Bru06}, and Klee and Witzgall~\cite[Cor.~3]{KleWit68}.
\begin{proof}Using (ii) of Theorem~\ref{facedim}, and the fact that $\{u\}$ is a face $\P(G,b)$ with dimension~0,
it follows that $0=|\supp(u)|-|V|+B'$, where $B'$ is the number of bipartite components of $\gr(u)$.  The result now follows from
the fact that $B'\ge B$ (since $\gr(u)$ is a spanning subgraph of $G$).

Alternatively, this result follows from~\eqref{incnulllem2} and Corollary~\ref{vertcor}
(again taking $P$, $M$, $N$, $A$ and $a$ to be $\P(G,b)$,~$V$,~$E$,~$I_G$ and~$b$ respectively).
\end{proof}

\begin{theorem}\label{vertth}
An element $u$ of $\P(G,b)$ is a vertex of $\P(G,b)$ if and only if
each component of the graph of $u$ either is acyclic or else contains exactly one cycle with that cycle having odd length.
\end{theorem}
Note that in this theorem, as in Proposition~\ref{zeronullity}, the choice of conditions for the components of~$G$ applies independently to each component.
As also indicated after Proposition~\ref{zeronullity}, the condition of this theorem can be restated as the condition
that~$G$ has no even-length cycles and no component containing more than one odd-length cycle.

A case of this theorem applied to $\N_{\le Z}(R)$ is given by, for example, Brualdi~\cite[Thm.\ 8.2.6]{Bru06},
and a case applied to $\N(R)$ is given by, for example,
Brualdi~\cite[Thm.~3.1]{Bru76},~\cite[Thm.~8.2.1]{Bru06}, and Lewin~\cite[Thm.~2]{Lew77}.
\begin{proof}
Using~\eqref{IMF},~\eqref{grx} and Proposition~\ref{vertlem}
(again taking $P$, $M$, $N$, $A$ and $a$ to be $\P(G,b)$, $V$, $E$, $I_G$ and~$b$ respectively,
so that $A'=I_{\gr(u)}$), it follows that $u$ is a vertex of $\P(G,b)$ if and only
if $\nullity(I_{\gr(u)})=0$.  The condition on $\gr(u)$ is then given by Proposition~\ref{zeronullity}
(taking the graph in that proposition to be $\gr(u)$).
\end{proof}
Theorem~\ref{vertth} can also be proved more directly, using arguments from the alternative proof of
Proposition~\ref{zeronullity}.  This will now be outlined briefly.\\[2.5mm]
\emph{Alternative proof of Theorem~\ref{vertth}.}  In this proof, the fact is used that an
element~$u$ of a polytope $P\subset\R^N$ (for a finite set~$N$) is a vertex of $P$ if and only if there does not exist any
nonzero $x\in\R^N$ such that $u-x\in P$ and $u+x\in P$.

First, consider $u\in\P(G,b)$ and $x\in\R^E$ such that $u\pm x\in\P(G,b)$, and
let each component of~$\gr(u)$ either be acyclic or else contain exactly one
cycle with that cycle having odd length. Then $I_G\,u=I_G(u\pm x)=b$, and so $I_G\,x=0$.  Also,
$u_e\pm x_e\ge0$ for each $e\in E$, and so if $u_e=0$ then $x_e=0$, i.e., $\supp(x)\subset\supp(u)$.
Therefore, $I_{\gr(u)}\,y=0$, where $y\in\R^{\supp(u)}$ is given by
$y_e=x_e$ for each $e\in\supp(u)$.
Using the same argument as in the first part of the alternative proof of Proposition~\ref{zeronullity},
it follows that $y=0$.  Hence, $x=0$ and $u$ is a vertex of $\P(G,b)$.

Now, conversely, consider $u\in\P(G,b)$ and let it not be the case that each component
of~$\gr(u)$ is acyclic or contains exactly one cycle with that cycle having odd length.
Then using the same argument as in the second part of the alternative proof of Proposition~\ref{zeronullity}, there exists a nonzero
$y\in\R^{\supp(u)}$ which satisfies $I_{\gr(u)}\,y=0$.  Now define $x\in\R^E$ by $x_e=y_e$ for each $e\in\supp(u)$, and
$x_e=0$ for each $e\in E\setminus\supp(u)$.  Then $x\ne0$, $I_G\,x=0$ and $\supp(x)\subset\supp(u)$.  Therefore, an
$\epsilon>0$ can be chosen such that $u_e\pm\epsilon x_e\ge0$ for all $e\in E$.  (In particular,
$0<\epsilon\le\min(\{u_e/x_e\mid e\in E,\;x_e>0\}\cup\{-u_e/x_e\mid e\in E,\;x_e<0\})$.)  Hence, $u\pm\epsilon x\in\P(G,b)$,
and $u$ is not a vertex of $\P(G,b)$.\hspace{\fill}$\Box$
\begin{corollary}\label{bipvertcor}
Let $G$ be bipartite.  Then an element $u$ of $\P(G,b)$ is a vertex of $\P(G,b)$ if and only if the graph of $u$
is a forest.\end{corollary}
A case of this result applied to $\N_{\le Z}(R,S)$ is given by, for example, Brualdi \cite[Thm.\ 8.1.10]{Bru06},
and a case applied to $\N(R,S)$ is given by, for example,
Brualdi~\cite[Thm.~8.1.2]{Bru06}, Klee and Witzgall~\cite[Thm.~4]{KleWit68}, and Schrijver~\cite[Thm.~21.15]{Sch03}.
\begin{proof}This result follows from Theorem~\ref{vertth}, using the fact that a bipartite graph does not contain
any odd-length cycles.\end{proof}

\begin{theorem}\label{explvert}
Let $H$ be the graph of a vertex $u$ of $\P(G,b)$.
Then $u$ is the only element of $\P(G,b)$ whose graph is $H$, and it is given explicitly by
\begin{equation}\label{ue}
u_e=\begin{cases}\displaystyle k_e\sum_{v\in V_{H\setminus e}(t_e)}(-1)^{d_{H\setminus e}(v,t_e)}\,b_v,&e\text{ is an edge of }H,\\
0,&\text{otherwise},\end{cases}
\end{equation}
for each $e\in E$, where
\begin{align}k_e&=\begin{cases}\frac{1}{2},&e\text{ is an edge of a nonloop cycle of $H$,}\\
1,&\text{otherwise,}\end{cases}\\
\label{teH}t_e&=\begin{cases}\mbox{the endpoint of $e$ furthest from }L,&e\text{ is an edge of a component}\\[-1.5mm]
&\text{of }H\text{ that contains a single}\\[-1.5mm]
&\text{cycle }L,\text{ but }e\text{ is not in }L,\\
\text{an arbitrarily-chosen endpoint of }e,&\text{otherwise,}\end{cases}\end{align}
$H\setminus e$ is the graph obtained by deleting edge~$e$ from $H$,
$V_{H\setminus e}(t_e)$ is the vertex set of the component of $H\setminus e$ which contains $t_e$,
and $d_{H\setminus e}(v,t_e)$ is the length of the (necessarily unique) path between~$v$ and~$t_e$ in $H\setminus e$.
\end{theorem}
The reasons for the uniqueness of the path between~$v$ and~$t_e$ in~$H\setminus e$, and for the
independence of the RHS of~\eqref{ue} on any choices of $t_e$ in~\eqref{teH}, will be given in the following proof.

Note also that the fact that if $H$ is the graph of a vertex $u$ of $\P(G,b)$,
then $u$ is the only element of $\P(G,b)$ with graph $H$ will also be given as part of Theorem~\ref{vertth2}.
\begin{proof}
Denote the edge set of $H$ as $E'$ (i.e., $E'=\supp(u)$),
and define $u'\in\R^{E'}$ by $u'_e=u_e$ for each $e\in E'$.
It can be seen that $I_H\,u'=b$, so it follows from Proposition~\ref{bipsollem} (taking~$G$ and $a$ in that proposition to be $H$ and $b$)
that $\sum_{v\in U_C}b_v=\sum_{v\in W_C}b_v$ for each bipartite component~$C$ of~$H$, where
$(U_C,W_C)$ is a bipartition for~$C$.
Also, since $u$ is a vertex of $\P(G,b)$, it follows from Theorem~\ref{vertth} that each
component of $H=\gr(u)$ either is acyclic or else contains exactly one cycle with that cycle having odd length.
It now follows from both parts of Proposition~\ref{uniquesollem} (taking~$G$ and $a$ in that proposition also to be $H$ and $b$), that
$u'$ is the only vector in $\R^{E'}$ with $I_H\,u'=b$, and that it is given explicitly by the RHS of~\eqref{xe} (which matches the
first case on the RHS of~\eqref{ue}).
Furthermore, as discussed after the statement of Proposition~\ref{uniquesollem},
the path between~$v$ and~$t_e$ in~$H\setminus e$ is unique, and the expression for~$u'$ is
independent of any choices of $t_e$ in~\eqref{teH}.
Therefore, $u$ is given by~\eqref{ue}, since
$u_e=u'_e$ for each $e\in E'$ and $u_e=0$ for each $e\in E\setminus E'$.
It can also be seen that the uniqueness of $u'$ as a vector in $\R^{E'}$ with $I_H\,u'=b$ implies
that $u$ is the only vector in~$\R^E$ with $I_G\,u=b$ and support $E'$, and hence that $u$ is the only element of
$\P(G,b)$ with graph~$H$.
\end{proof}

\begin{theorem}\label{edgeth}
Let $u$ and $w$ be distinct vertices of $\P(G,b)$.
Then~$u$ and~$w$ are the vertices of an edge of $\P(G,b)$ if and only if
$\gr(u)\cup\gr(w)$ has a component~$C$ such that
each component of~$\gr(u)\cup\gr(w)$ other than $C$ either is acyclic or else contains exactly one cycle with that cycle having odd length,
while the cycle content of $C$ is one of the following:
\begin{itemize}
\item exactly one cycle, with that cycle having even length, or
\item exactly two cycles, with at least one of those cycles having odd length, or
\item exactly one even-length and exactly two odd-length cycles, with any two of those cycles sharing at least one edge.
\end{itemize}
\end{theorem}
Note that $\gr(u)\cup\gr(w)$ is the graph union~\eqref{graphunion} of $\gr(u)$ and $\gr(w)$,
and that, using~\eqref{grXgrx}, $\gr(u)\cup\gr(w)=\gr(\{u,w\})$.
\begin{proof}
Using~\eqref{IMF} and Proposition~\ref{edgelem}
(again taking $P$, $M$, $N$, $A$ and $a$ to be $\P(G,b)$, $V$, $E$, $I_G$ and~$b$ respectively,
so that $A'=I_{\gr(u)\cup\gr(w)}$), it follows that~$u$ and~$w$ are the vertices of an edge of $\P(G,b)$ if and only
if $\nullity(I_{\gr(u)\cup\gr(w)})=1$.  The condition on $\gr(u)\cup\gr(w)$ is then given by Proposition~\ref{onenullity}
(taking the graph in that proposition to be $\gr(u)\cup\gr(w)$).
\end{proof}

\begin{corollary}\label{bipedgeth}
Let $G$ be bipartite, and let $u$ and $w$ be distinct vertices of $\P(G,b)$. Then~$u$
and~$w$ are the vertices of an edge of $P$ if and only if $\gr(u)\cup\gr(w)$ contains exactly one cycle.
\end{corollary}
A case of this theorem applied to $\N(R,S)$ is given by, for example, Brualdi~\cite[Thm.\ 8.4.6]{Bru06},
Oviedo~\cite[Cor.~1]{Ovi96}, and Yemelichev, Kovalev and Kravtsov~\cite[Ch.~6, Lem.~4.1]{YemKovKra84}.

\begin{proof}This result follows from Theorem~\ref{edgeth}, using the fact that a bipartite graph does not contain any
 odd-length cycles.\end{proof}

\section{Further results for the vertices, edges, faces and graphs of $\mathcal{P}(G,b)$}\label{facelattPGb}
In this section, further results concerning the vertices, edges, faces and graphs of $\P(G,b)$ are obtained.
The previous such results, in Section~\ref{facesPGb}, combined
both general results for graphs, from Section~\ref{inc}, and general results for polytopes, from
Section~\ref{poly}.  However, by contrast, the results of this section depend only on
general results for polytopes, from Section~\ref{facelatt}, together with the simple correspondences,
as given in~\eqref{grX}--\eqref{grx}, between
supports of subsets or elements of~$\P(G,b)$, and graphs of~$\P(G,b)$.

In particular, this section consists of results from Section~\ref{facelatt}, in which the polytope $P$
of~\eqref{polytope} is now taken to be~$\P(G,b)$, in the form~\eqref{IMF}, with
$N$, $M$, $A$ and $a$ in~\eqref{polytope} taken to be~$E$,~$V$,~$I_G$ and~$b$ respectively,
and with unions, intersections or containments of supports of subsets or elements of~$\P(G,b)$
now expressed as unions, intersections or containments (as given in~\eqref{graphint}--\eqref{subgraph})
of graphs of~$\P(G,b)$.

It will be assumed throughout this section that $b$ is nonzero,
so that the application of the results of Section~\ref{facelatt} to~$\P(G,b)$ is valid.

Included among the results of this section are
further characterizations of the elements of~$\P(G,b)$ which are vertices of $\P(G,b)$ (in Theorem~\ref{vertth2}),
and of the pairs of distinct vertices of $\P(G,b)$ which form edges of $\P(G,b)$ (in Corollary~\ref{edgeth2}),
several equivalent conditions for a spanning subgraph of $G$ to be a graph of~$\P(G,b)$
(in Theorem~\ref{faceth}), and a statement that the set of graphs of $\P(G,b)$ forms a lattice which
is isomorphic to the face lattice of~$\P(G,b)$ (in Theorem~\ref{facelattth}).

\begin{proposition}
For any subset $X$ of $\P(G,b)$,
\begin{equation}\label{grXprop}\gr(X)=\gr\Biggl(\bigcap_{\substack{F\in\F(\P(G,b))\\X\subset F}}F\Biggr)
=\gr\Biggl(\bigcap_{\substack{F\in\facets(\P(G,b))\\X\subset F}}F\Biggr).\end{equation}
\end{proposition}
Note that $\F(\P(G,b))$ and $\facets(\P(G,b))$ are the face lattice and set of facets, respectively, of $\P(G,b)$,
using the notation of Section~\ref{facelatt}.
\begin{proof}
This result follows by applying~\eqref{suppX} to $\P(G,b)$.
\end{proof}
\begin{proposition}
For any subset $X$ of $\P(G,b)$, and any face $F$ of $\P(G,b)$,
\begin{equation}\label{XFprop}X\subset F\text{ if and only if }\gr(X)\subset\gr(F).\end{equation}
\end{proposition}
\begin{proof}
This result follows by applying~\eqref{suppiff} to $\P(G,b)$.
\end{proof}
Note that~\eqref{suppXx} and~\eqref{suppiff2} could also easily be applied
to $\P(G,b)$, giving special cases of~\eqref{grXprop} and~\eqref{XFprop}, respectively, in which $X$ contains a single element.

\begin{theorem}\label{vertth2}
Let $u$ be an element of $\P(G,b)$, and $H$ be the graph of $u$.  Then the following are equivalent.
\begin{list}{(\roman{listnumber})}{\usecounter{listnumber}\setlength{\labelwidth}{8mm}\setlength{\leftmargin}{10mm}\setlength{\itemsep}{1mm}}
\item $u$ is a vertex of $\P(G,b)$.
\item $u$ is the only element of $\P(G,b)$ whose graph is $H$.
\item $u$ is the only element of $\P(G,b)$ whose graph is contained in $H$.
\end{list}
\end{theorem}
Note that the the implication of (ii) by (i) in this theorem is also given in Theorem~\ref{explvert}.

Cases of this theorem applied to $\N_{\le Z}(R)$, $\N(R)$, $\N_{\le Z}(R,S)$ and~$\N(R,S)$ are given by,
for example, Brualdi~\cite[Thm.~3.1]{Bru76},~\cite[Thms.~8.1.2, 8.1.10, 8.2.1 \& 8.2.6]{Bru06},
Jurkat and Ryser~\cite[p.~348]{JurRys67}, and Klee and Witzgall~\cite[Cor.~2]{KleWit68}.
\begin{proof}This result follows by applying \eqref{vertcond}
to $\P(G,b)$.
\end{proof}

\begin{theorem}\label{faceth1}
Let $U$ be a subset of vertices of $\P(G,b)$.  Then
$U$ is the set of vertices of a face of $\P(G,b)$
if and only if the elements of $U$ are the only vertices of $\P(G,b)$ whose graphs are contained
in the graph of $U$.
\end{theorem}
\begin{proof}This result follows by applying \eqref{VP2}
to $\P(G,b)$.
\end{proof}

\begin{corollary}\label{edgeth2}
Let $u$ and $w$ be distinct vertices of $\P(G,b)$.  Then
$u$ and $w$ are the vertices of an edge of $\P(G,b)$
if and only if $u$ and $w$ are the only vertices of $\P(G,b)$ whose graphs are contained
in the union of the graphs of $u$ and $w$.
\end{corollary}
\begin{proof}This result follows from Theorem~\ref{faceth1} by taking $U$ to be $\{u,w\}$.  (It
also follows by applying~\eqref{edgecond} to $\P(G,b)$.)
\end{proof}

\begin{theorem}\label{faceth}
Let $H$ be a spanning subgraph of $G$.  Then the following are equivalent.
\begin{list}{(\roman{listnumber})}{\usecounter{listnumber}\setlength{\labelwidth}{8mm}\setlength{\leftmargin}{10mm}\setlength{\itemsep}{1mm}}
\item $H$ is a graph of $\P(G,b)$.
\item $H$ is the graph of a subset of $\P(G,b)$.
\item $H$ is the graph of an element of $\P(G,b)$, or $H$ has no edges.
\item $H$ is the graph of a face of $\P(G,b)$.
\item $H$ is a union of graphs of vertices of $\P(G,b)$.
\item $H$ is the union of the graphs of all vertices of $\P(G,b)$ whose graph is contained in~$H$.
\end{list}
\end{theorem}
Note that a seventh equivalent condition, which depends only on~$H$and~$b$, without
any reference to $\P(G,b)$, will be added to this list in Theorem~\ref{suppcond}.
\begin{proof}The equivalence of conditions (i) and (ii) is simply the definition of a graph of~$\P(G,b)$ (as given after~\eqref{GGbdef}).
The equivalence of condition (iv) and each of conditions~(ii),~(iii),~(v) or~(vi) follows by applying
the equality between the first set, $\S(P)$, in~\eqref{SP} (as defined in~\eqref{SPdef})
and each of the other four sets in~\eqref{SP} (using~\eqref{suppXsuppx} in the third of these), respectively, to~$\P(G,b)$.
\end{proof}
It follows from~\eqref{GGbdef} and Theorem~\ref{faceth} that the set of graphs of $\P(G,b)$ can now be written as
\begin{align}\notag\G(G,b)&=\{\gr(X)\mid X\subset\P(G,b)\}\\
\notag&=\{\gr(x)\mid x\in\P(G,b)\}\cup\{\emptyset\}\\
\notag&=\{\gr(F)\mid F\in\F(\P(G,b))\}\\
\notag&\textstyle=\{\bigcup_{u\in U}\gr(u)\mid U\subset\vert(\P(G,b))\}\\
\label{GGb}&\textstyle=\Bigl\{\text{spanning subgraphs }H\text{ of }G\Bigm|
H=\bigcup_{\substack{u\in\vert(\P(G,b))\\\gr(u)\subset H}}\gr(u)\Bigr\}.\end{align}
\begin{theorem}\label{facelattth}
The face lattice~$\F(\P(G,b))$ of $\P(G,b)$ is isomorphic to the set~$\G(G,b)$ of graphs of~$\P(G,b)$
partially ordered by inclusion.

The natural isomorphism between these lattices maps
each face $F\in\F(\P(G,b))$ to its graph $\gr(F)$, and inversely maps each graph
$H\in\G(G,b)$ to the face $\{x\in\P(G,b)\mid\gr(x)\subset H\}=
\{x\in\P(G,b)\mid x_e=0$ for each $e\in E$ which is not an edge of $H\}$.
In terms of vertices of faces, $F\in\P(G,b)$ is mapped to the union of the
graphs of the vertices of $F$, i.e., $\gr(F)=\bigcup_{u\in\vert(F)}\gr(u)=\gr(\vert(F))$, and
$H\in\G(G,b)$ is mapped to the face whose vertices are $\{u\in\vert(P)\mid\gr(u)\subset H\}$.

For any $\H\subset\G(G,b)$, the meet of~$\H$ is the union of all those graphs of $\G(G,b)$
which are contained in each graph in~$\H$, or alternatively the union of the graphs of all those vertices of~$\P(G,b)$
whose graphs are contained in each graph in~$\H$,
and the join of $\H$ is the union of all the graphs in $\H$, i.e.,
\begin{align}\label{infH}\inf(\H)&=\bigcup_{\substack{H\in\G(G,b)\\H\subset\cap_{H'\in\H}H'}}\!\!H\;
=\bigcup_{\substack{u\in\vert(\P(G,b))\\\gr(u)\subset\cap_{H'\in\H}H'}}\!\gr(u),\\
\sup(\H)&=\bigcup_{H\in\H}H.\end{align}
\end{theorem}
Note that, for the case $\H=\emptyset$ in~\eqref{infH},
$\bigcap_{H'\in\mathcal{\emptyset}}H'$ can be taken as $G$ or as $\gr(\P(G,b))$, giving $\inf(\emptyset)=\gr(\P(G,b))$.

Note also that the dimension of a nonempty face $F$ of $\P(G,b)$ is given by (ii) of Theorem~\ref{facedim},
with $|\supp(F)|$ in that theorem being simply the number of edges in the graph of $F$.
\begin{proof}All of these results follow from the discussion, in Section~\ref{facelatt}, between~\eqref{SP} and Figure~\ref{iso},
as applied to $\P(G,b)$.
\end{proof}

\section{Further conditions for the graphs of $\mathcal{P}(G,b)$}\label{furthcond}
In this section, additional conditions to those of Theorem~\ref{faceth}, for a spanning subgraph of~$G$ to be a graph of~$\P(G,b)$, are
obtained using Theorems~\ref{bipposth} and~\ref{posth}
from Section~\ref{nonempt}. In contrast to the conditions of Theorem~\ref{faceth}, the conditions of this section depend only
on the spanning subgraph and $b$, and take the form of finitely-many strict inequalities and equalities for
certain sums of entries of $b$.

It is assumed in this section that $b$ is again nonzero.

\begin{theorem}\label{suppcond}
Let $H$ be a spanning subgraph of $G$.  Then the following
condition is equivalent to the six conditions of Theorem~\ref{faceth}.
\begin{list}{(vii)}{\setlength{\labelwidth}{8mm}\setlength{\leftmargin}{10mm}\setlength{\itemsep}{1mm}}
\item $H$ has no edges, or $\sum_{v\in V_1}b_v\ge\sum_{v\in V_3}b_v$ for all
sets $V_1$, $V_2$ and $V_3$ such that $V=V_1\cupdot V_2\cupdot V_3$ and $H[V_2\cup V_3,V_3]=\emptyset$,
with the inequality holding as an equality if and only if $H[V_1,V_1\cup V_2]=\emptyset$.
\end{list}
\end{theorem}
A case of this theorem applied to~$\N(R)$ is given by Brualdi~\cite[p.~353]{Bru06}.
\begin{proof}
It will be shown that the condition (vii) of this theorem is equivalent to condition~(iii) in Theorem~\ref{faceth}.
The equivalence of all seven conditions then follows from Theorem~\ref{faceth}.

Denote the edge set of $H$ as $E'$.  If $E'=\emptyset$, then conditions~(iii) and~(vii) are both automatically satisfied.

So, assume that $E'\ne\emptyset$ and that condition~(iii) is satisfied.  Then there exists $x\in\P(G,b)$ with $\gr(x)=H$.
Now define $x'\in\R^{E'}$ by $x'_e=x_e$ for each $e\in E'$.  It can be seen that $x'$ is a strictly positive element of $\P(H,b)$.
Therefore, using Theorem~\ref{posth} (with its graph taken to be~$H$), condition (vii) is satisfied.

Conversely, assume that $E'\ne\emptyset$ and that condition~(vii) is satisfied.  Then, using Theorem~\ref{posth}
(with its graph again taken to be~$H$), $\P(H,b)$ contains a strictly positive element~$x'$.
Now define $x\in\R^E$ by $x_e=x'_e$ for each $e\in E'$, and $x_e=0$ for each $e\in E\setminus E'$.
It can be seen that $x\in\P(G,b)$ and $\gr(x)=H$. Therefore, condition~(iii) is satisfied.
\end{proof}
It follows from Theorem~\ref{suppcond} that the equalities of~\eqref{GGb}, for the set of graphs of $\P(G,b)$, can now be supplemented by
\begin{equation}\label{GGb2}
\G(G,b)=\bigl\{\text{spanning subgraphs }H\text{ of }G\;\big|\;H\text{ satisfies (vii) in Theorem~\ref{suppcond}}\bigr\}.\end{equation}

\begin{theorem}\label{bipsuppcond}Let $H$ be a spanning subgraph of $G$. If $H$ is bipartite, then
the following condition is equivalent to the six conditions of Theorem~\ref{faceth}, and to the condition (vii) in
Theorem~\ref{suppcond}.
\begin{list}{(viii)}{\setlength{\labelwidth}{9mm}\setlength{\leftmargin}{11mm}\setlength{\itemsep}{1mm}}
\item $H$ has no edges, or $\sum_{v\in C}b_v\ge\sum_{v\in V\setminus C}b_v$ for each vertex cover~$C$ of $H$,
with the inequality holding as an equality if and only if $V\setminus C$ is also a vertex cover of~$H$.
\end{list}
\end{theorem}
Note that if $(U,W)$ is a bipartition for $H$, then
condition~(viii) in this theorem is equivalent to the condition that $H$ has no edges, or
$\sum_{v\in U_1}b_v+\sum_{v\in W_1}b_v\ge\sum_{v\in U_2}b_v+\sum_{v\in W_2}b_v$ for all sets~$U_1$,~$U_2$,~$W_1$ and~$W_2$
such that $U=U_1\cupdot U_2$, $W=W_1\cupdot W_2$ and $H[U_2,W_2]=\emptyset$, with the inequality holding
as an equality if and only if $H[U_1,W_1]=\emptyset$.
Also, this condition remains unchanged if its inequality is replaced by
$\sum_{v\in U_1}b_v\ge\sum_{v\in W_2}b_v$, or by $\sum_{v\in W_1}b_v\ge\sum_{v\in U_2}b_v$.
The reasons for these equivalences are discussed briefly after the statements of Theorems~\ref{bipnonemptth}
and~\ref{bipposth} (where the graph in those remarks should now be taken to be~$H$).
\begin{proof}
It follows from Theorems~\ref{bipposth} and~\ref{posth} (taking the graph in each theorem to be $H$) that if~$H$ is bipartite, then
condition~(vii) in Theorem~\ref{suppcond} and condition~(viii) in this theorem are equivalent.  The equivalence of these two conditions to
the six conditions in Theorem~\ref{faceth} is then given by Theorem~\ref{suppcond}.
\end{proof}

\begin{corollary}\label{bipsuppcond1}
Let $G$ be bipartite, and $H$ be a spanning subgraph of $G$. Then
the six conditions of Theorem~\ref{faceth}, condition (vii) in
Theorem~\ref{suppcond}, and condition (viii) in Theorem~\ref{bipsuppcond} are all equivalent.
\end{corollary}
A case of this theorem applied to~$\N(R,S)$ is given by Brualdi~\cite[p.~343]{Bru06}.
\begin{proof}
This result follows immediately from Theorem~\ref{bipsuppcond}, since any spanning subgraph of a bipartite graph is bipartite.
\end{proof}

\end{document}